\documentclass[10pt,twocolumn,twoside,dvipsnames]{IEEEtran} 

\usepackage{cite}
\usepackage{amsmath,amssymb,amsfonts}
\usepackage{graphicx}
\usepackage{algorithmic}
\usepackage{hyperref}
\usepackage{textcomp}

\usepackage{enumerate}
\usepackage{wrapfig,lipsum}
\usepackage{bbm}
\usepackage{tikz}
\usepackage{soul}
\usepackage{cancel}
\usepackage{amsthm}

\usepackage[bb=boondox]{mathalfa}
\usetikzlibrary{decorations.pathmorphing}

\newcommand{\Ec}{\mathcal{E}}

\newcommand{\Rc}{\mathcal{R}}

\newcommand{\Sc}{\mathcal{S}}
\newcommand{\Pc}{\mathcal{P}}
\newcommand{\Zc}{\mathcal{Z}}

\newcommand{\Uc}{\mathcal{U}}
\newcommand{\Vc}{\mathcal{V}}

\newcommand{\im}{{\rm Im}}

\DeclareMathOperator{\rank}{rank}

\definecolor{darkviolet}{rgb}{0.58, 0.0, 0.83}
\definecolor{lavender}{rgb}{0.45, 0.31, 0.59}

\usetikzlibrary{arrows}
\newtheorem{lemma}{Lemma}
\newtheorem{remark}{Remark}

\newtheorem{theorem}{Theorem}
\newtheorem{definition}{Definition}

\newtheorem{proposition}{Proposition}
\newtheorem{corollary}{Corollary}

\newcommand{\myend}{\hfill $\blacktriangle$}

\colorlet{inputcolor}{Orchid!30}

\begin{document}
\title{Obtaining Structural Network Controllability\\with Higher-Order Local Dynamics}
\author{Marco Peruzzo, Giacomo Baggio, Francesco Ticozzi
\thanks{The authors are with the Department of Information Engineering, University of Padova, Italy. Emails: \texttt{marco.peruzzo.5@phd.unipd.it}, \texttt{baggio@dei.unipd.it}, \texttt{ticozzi@dei.unipd.it}. }}
\maketitle
\begin{abstract} 
We consider a network of identical, first-order linear systems, and investigate how replacing a subset of the systems composing the network with higher-order ones, either taken to be generic or specifically designed, may affect its controllability. 
After establishing a correspondence between state controllability in networks of first-order systems with output controllability in networks of higher-order systems, we show that adding higher-order dynamics may require significantly fewer subsystem modifications to achieve structural controllability, when compared to first-order heterogeneous subsystems. Furthermore, we characterize the topology of networks (which we call X-networks) in which the introduction of heterogeneous local dynamics is not necessary for structural output controllability, as the latter can be attained by suitable higher-order subsystems with homogeneous internal dynamics.
\end{abstract}
\begin{IEEEkeywords}
Linear network systems, structural controllability, higher-order dynamics
\end{IEEEkeywords}

\vspace{-0.15cm}

\section{Introduction}\label{sec:introduction}

Many natural and engineered dynamical systems can be modeled as networks consisting of a large number of interconnected, simpler dynamical units \cite{MN:03,SB-VL-YM-MC-DH:06,FB:24}. In such systems, the interconnection structure significantly influences the collective behavior of the system and the response to external inputs. As a result, understanding how network topology affects dynamics is fundamental to the analysis and control of these systems. 

Over the past few decades, the analysis of complex network dynamics has gained increasing attention from different scientific communities, ranging from physics and engineering \cite{GP-MA:13,GC-TS-FS-DG-AG-GC:19} to biology and social sciences \cite{TV-KR-LA:06,AP-RT:17}. This growing body of research has deepened our understanding of how the structural properties of networks impact collective phenomena in interconnected systems, with notable implications for synchronization \cite{FD-FB:14}, information transfer \cite{GB-SZ:21},~and~controllability~\cite{LY-BA:2016}.

In particular, the challenge of efficiently controlling large-scale networks has driven renewed interest in the structural controllability analysis of networks governed by linear time-invariant dynamics \cite{YYL-JJS-ALB:11}. Structural controllability is a relaxed notion of controllability that depends only on the sparsity patterns of system matrices, rather than their specific numerical values \cite{CTL:74}. For network systems, it highlights the role of the interconnection structure in determining whether the system state can be steered to any desired configuration via appropriate inputs, regardless of the precise values of the edge weights. Although structural controllability has been extensively studied in control theory \cite{KJR:88, JMD-CC-JV:03,GR-APA-SP:22}, recent works have opened up new areas of investigation. These include: developing strategies to efficiently achieve structural controllability \cite{YYL-JJS-ALB:11}; understanding how key network properties, such as degree distribution and modularity, influence structural controllability \cite{MP-YYL-JJS-ALB:13,GM-LD-GB:14}; and analyzing structural controllability in time-varying settings \cite{AL-SPC-YYL-LW-ALB:17, LL-FL-CA:24} or under constraints on edge weights \cite{YG-AL-LW:21,FL-ASM:19}. Regarding the first area, much of existing research has focused on developing efficient methods to identify a minimal set of driver nodes \cite{YYL-JJS-ALB:11, CC-JMD:13, SP-SK-AAP:15}, or to implement minimal changes to the network topology \cite{WXW-XN-YCL-CG:12, XC-SP-GJP-VMP:18, SM-PC-MNB:20} to guarantee structural controllability.

In this work, we explore an alternative strategy to achieve structural controllability, namely through the introduction of higher-order local dynamics in a subset of network nodes.\footnote{ We stress that, in this paper, the term `higher-order' refers to the number of state variables describing the local dynamics of network nodes, not to the nature of interactions between nodes as in \cite{CB-EG-HAH-MTS:23}.}
More precisely, we consider networks composed of identical first-order systems and replace the internal dynamics of certain nodes with higher-order ones, while preserving the original configuration 
of interconnections between nodes. We assume that a generic linear combination of the internal states of each modified node is measured, and we analyze the structural output controllability \cite{KM-SP:90} of the resulting network system. First, we characterize the topologies of networks which are structurally uncontrollable and can be made structurally {(output) controllable 
by replacing the internal dynamics of some nodes with homogeneous higher-order ones. 
 The existence of such systems, which we name X-networks, already clearly demonstrates the potential of higher-order systems, as heterogeneous local dynamics is needed to achieve structural controllability using first-order systems \cite{CC-AK:19}. Second, we identify a class of networks for which higher-order homogeneous dynamics cannot lead to structural {output} controllability, referred to as Y-networks. We then show that, for Y-networks, using heterogeneous higher-order local dynamics provides an advantage in the minimal number of modified nodes, compared to employing first-order heterogeneous dynamics. Combining the results, we have that structural output controllability can be attained by introducing a reduced number of homogeneous and of heterogeneous higher-order dynamics. As a by-product of our analysis, we establish alternative graphical characterizations of structurally uncontrollable network topologies and present conditions for output controllability of linear systems based on the PBH test, which complement the results in \cite{BD-JL-MJ:23}.

Finally, we mention that a closely related work to ours is \cite{CC:19}, where the author provides conditions for the structural controllability of interconnected subsystems with potentially higher-order internal dynamics. However, in \cite{CC:19} the subsystems are assumed to be single-input single-output, a condition we do not impose in our work.
Another related work is \cite{MJ-JL:22}, which considers  multi-input multi-output higher-order subsystems, provides conditions for the structural controllability of the network, and discusses  how to improve controllability via topology design. We emphasize that, in  our work, the network topology remains fixed, and the designer's degrees of freedom are limited to selecting which first-order subsystems to replace with higher-order ones. Furthermore  and more importantly, both works \cite{CC:19, MJ-JL:22} do not associate a (scalar) output equation with the nodes of a subsystem, as we do here, and focus on the structural (state) controllability~of~the~global~network.

The rest of the paper is organized as follows. In Sec.~\ref{sec:models-and-problems}, we introduce the models of both the original and extended network systems, the latter incorporating higher-order nodes, and formally define the problem addressed in this study. Sec.~\ref{sec:output-controllability-modified} presents some instrumental results on the output structural controllability of the extended system. The main findings of the paper, namely the characterization of the controllability advantages provided by higher-order dynamics, are discussed in Sec.~\ref{sec:advantages}. Finally, in Sec.~\ref{sec:conclusion} we present some concluding remarks and outline potential directions for future research.

\noindent \textbf{Notation.} We denote with $\mathbb{R}^{n\times m}$ the set of $n\times m$ matrices with real entries. 
We let $A^\top$ denote the transpose of a matrix $A\in \mathbb{R}^{n\times m}$ and $\ker(A)$ be its kernel.
The symbol $I_n$ stands for the identity matrix in $\mathbb{R}^{n\times n}$, we will drop the subscript $n$ when the dimension is clear from the context. Given a set of vectors \(\{w_i\}\) in \(\mathbb{R}^n\), the notation \(\operatorname{span}\{w_1, \dots, w_n\}\) represents their linear span. Finally, we use calligraphic letters to denote sets, and for a set \(\mathcal{Z}\), the symbol \(|\mathcal{Z}|\) indicate its cardinality.

\section{Models and problem}\label{sec:models-and-problems}

\subsection{Networks of identical first-order systems}

We consider a network system composed of $n$ interconnected subsystems with identical (homogeneous) first-order internal dynamics. The interconnections among these subsystems are represented by a directed graph ${\rm G}_0=(\mathcal{V},{\rm \Ec}_v)$, where $\mathcal{V}$ and ${\Ec}_v$ are the set of nodes and edges of ${\rm G}_0$. If a subsystem is connected to another, it influences its dynamics with a linear function of its state acting as an input for the second. The network is also influenced by an external control input $u(t)\in\mathbb{R}^m$.
The $i$-th subsystem is then governed by the linear and time-invariant dynamics
\begin{align}\label{eq:subsys}
    \dot{x}_i(t)= a x_i(t) + \sum_{j\ne i} a_{ij}x_{j}(t) + b_i u(t)
\end{align}
where $x_i(t)\in\mathbb{R}$ is the state of the $i$-th subsystem,  $a\in\mathbb{R}$ defines the internal dynamics of the subsystem, $a_{ij}=0$ if $(j,i)\not\in\mathcal{E}_v$, and the local input matrix $b_i\in\mathbb{R}^{1\times m}$ determines whether the external input directly affects the subsystem. 

We can express the network dynamics in standard vector form by introducing also an output equation as follows
\begin{align}\label{eq:netsys}
\Sigma: \ \begin{cases}
    \dot{x}(t)=
     A x(t) + B u(t), & \\ 
    y(t) =  C x(t) & 
\end{cases}
\end{align}
where $x(t)=\begin{bmatrix} x_1(t) \cdots\,  x_n(t)\end{bmatrix}^{\top}\in\mathbb{R}^n$ and $y(t)=\begin{bmatrix} y_1(t) \cdots\,  y_p(t)\end{bmatrix}^{\top}\in\mathbb{R}^p$ denote the state and output of the network system, respectively. The diagonal entries of  $A\in\mathbb{R}^{n\times n}$ describe the internal dynamics of subsystems while the off-diagonal entries of $A$ capture the interconnections among subsystems, namely $A_{ii}=a$ and $A_{ij}=a_{ij}$ for $i\ne j$. The input matrix $B$ is given by $B = \begin{bmatrix}b_1^\top & \cdots & b_n^\top\end{bmatrix}^\top \in\mathbb{R}^{n\times m}$. 

Throughout the paper, we assume that $p=n$ and the output matrix $C$ is the identity matrix, implying that all states are directly measurable. Furthermore, we assume that subsystems in $\Sigma $ have trivial internal dynamics\footnote{{To simplify the analysis, we assume trivial internal dynamics for the network subsystems. However, our results extend to networks of identical first-order systems with non-trivial internal dynamics ($a \neq 0$). 
 This follows from the fact that the controllability matrices of \eqref{eq:netsys} with $a\ne 0$ and $a=0$ have the same rank for every choice of the free parameters, see e.g.~\cite{CC-AK:19}.}} ($a=0$) and we look at \eqref{eq:netsys} as a \textit{structured system}, that is, the interconnection weights $a_{ij}$, with $(j,i)\in\mathcal{E}_v$, $i\ne j$, and the non-zero entries of $B$ can be chosen as arbitrary real numbers.

\subsection{Structural controllability with extended internal dynamics}\label{sec:hod}

Our aim is to improve the output controllability of the structured network system \eqref{eq:netsys}.  We recall that a structured linear system $\Sigma$ as in \eqref{eq:netsys} is (i) \emph{structurally controllable} if there exists a numerical realization of $A$, $B$ such that the controllability matrix $\Rc = [B\ AB\ \cdots \ A^{n-1}B]$ has full row-rank, and (ii) \emph{structurally output controllable} if there exists a numerical realization of $A$, $B$, $C$ such that the output controllability matrix $\Rc_o = C\Rc$ has full row-rank. 

In some cases we consider only $A$ and $B$ to be structured matrices, while $C$ is taken to be a fixed matrix. In particular, for \eqref{eq:netsys} with $C=I,$ structural output controllability is equivalent to structural controllability. 

While prior studies have approached the problem of improving structural (output) controllability by optimizing input selection \cite{ YYL-JJS-ALB:11,CC-JMD:13,SP-SK-AAP:15} or adjusting the network topology \cite{WXW-XN-YCL-CG:12,XC-SP-GJP-VMP:18,SM-PC-MNB:20}, our approach is to {\em attain controllability by modifying the internal dynamics of the~individual~subsystems}~\eqref{eq:subsys}. 

Specifically, we shall consider replacing the identical, first-order (some or all) subsystem dynamics of \eqref{eq:netsys} with:
\begin{itemize}
\item {\em Heterogeneous dynamics}, allowing different subsystems to feature different internal dynamics; 
\item{\em Higher-order dynamics}, allowing subsystems to have internal dynamics of order greater than one.
\end{itemize}
Namely, we replace \eqref{eq:subsys} with
\begin{align}\label{eq:subsys-exp}
    \dot{\hat{x}}_i(t)= \hat{a}_i x_i(t) + \sum_{j\ne i} \hat{a}_{ij}\hat{x}_{j}(t) + \hat{b}_i u(t)
\end{align}
where $\hat{x}_i(t)\in\mathbb{R}^{\hat{n}_i}$ is the state of the $i$-th modified subsystem with $\hat{n}_i\ge 1$ being its order. The matrix $\hat{a}_i\in\mathbb{R}^{\hat{n}_i\times \hat{n}_i}$ describes the new (potentially higher-order) internal dynamics of the $i$-th subsystem, while $\hat{a}_{ij}\in\mathbb{R}^{\hat{n}_i\times \hat{n}_j}$ the new interconnections weights which satisfy the original locality constraints, that is $\hat{a}_{ij}=0$ if $(j,i)\not\in\mathcal{E}_v$. Lastly, $\hat{b}_i\in\mathbb{R}^{\hat{n}_i\times m}$ denotes the new local input matrix, which satisfies $\hat{b}_i= 0$ if $b_i=0$.

After modifying the original subsystems \eqref{eq:subsys} as in \eqref{eq:subsys-exp}, the overall network dynamics can be expressed as
\begin{align}\label{eq:netsys-exp}
\hat{\Sigma}: \ \begin{cases}
    \dot{\hat{x}}(t)=
     \hat{A} \hat{x}(t) + \hat{B} u(t), & \\ 
    \hat{y}(t) =  \hat{C} \hat{x}(t) & 
\end{cases}
\end{align}
where $\hat{x}(t)=\begin{bmatrix} \hat{x}^{{\top}}_1(t) \cdots\,  \hat{x}^{{\top}}_n(t)\end{bmatrix}^{\top}\in\mathbb{R}^{\hat{n}}$, $\hat{y}(t)=\begin{bmatrix} \hat{y}_1(t) \cdots\,  \hat{y}_n(t)\end{bmatrix}^{\top}\in\mathbb{R}^{{n}}$,  $\hat{n}=\sum_{i=1}^n \hat{n}_i$, denote the state and output of the extended network, respectively. The state matrix $\hat{A}\in\mathbb{R}^{\hat{n}\times \hat{n}}$, input matrix $\hat{B}\in\mathbb{R}^{\hat{n}\times m}$ and output matrix $\hat{C}\in\mathbb{R}^{n\times \hat{n}}$ have the following block structure
\begin{align*}
    \hat{A} = \begin{bmatrix}\hat{a}_{1} &\cdots& \hat{a}_{1n} \\ \vdots &\ddots& \vdots \\ \hat{a}_{n1} &\cdots& \hat{a}_{n}\end{bmatrix}, \   \hat{B} = \begin{bmatrix} \hat{b}_1 \\ \vdots\\ \hat{b}_{n}\end{bmatrix}, \  \hat{C} = \begin{bmatrix}\hat{c}_{1} &\cdots& 0 \\ \vdots &\ddots& \vdots \\ 0 &\cdots& \hat{c}_{n}\end{bmatrix},
\end{align*}
where $\hat{c}_i\in\mathbb{R}^{1\times \hat{n}_i}$. 
From now on, we consider $\hat{\Sigma}$ as a \textit{structured system} in which a key constraint is imposed to preserve the nature of the interconnections of $\Sigma$: $\hat{a}_{ij}$, $i\ne j$, is allowed to be a generic {\em 
matrix} {\em only if} ${a}_{ij}\ne0$. This implies that the communication between interconnected subsystems is possible only if it was already allowed in model \eqref{eq:netsys}.
The remaining entries of the state matrix, described by $\hat{a}_i$, are either generic or \textit{all} fixed to zero. In the former case, subsystems may have different internal dynamics and are therefore \textit{heterogeneous}, in the latter case all the network subsystems have trivial internal dynamics and are said to be \textit{homogeneous}. 
{Moreover, we let $\hat{b}_{i,s}$, ${b}_{i,s}$ be respectively the $s$-th column of $\hat{b}_i$, $b_i$ and we let all the entries of  $\hat{b}_{i,s}$ be generic only if ${b}_{i,s}\neq 0$. } This ensures that only the input components that affected directly a node $i$ in the original system are allowed to enter the dynamics of the extended $i$-th subsystem.

Similarly, we assume that $\hat{c}_i$ have generic real entries if $\hat{n}_i>1$, meaning that the output $\hat{y}_i(t)$ is an arbitrary linear combination of the components of the internal state $\hat{x}_i(t)$, and $\hat{c}_i=1$ otherwise.

\begin{figure}
    \centering
    {\resizebox{0.425\textwidth}{!}{
	\begin{tikzpicture}[shorten >=1pt, auto, ultra thick,
		node_style/.style={draw, circle,thick, fill=white, minimum size=0.5cm,font=\footnotesize},every edge/.append style = {thick}]

		\begin{scope}
            \node[node_style,fill=inputcolor,inner sep=0.01cm] (u) at (0,1) {$u$};
			\node[node_style,inner sep=0.01cm] (1) at (0,0) {$x_1$};
			\node[node_style,inner sep=0.01cm] (3) at (0.5,-1) {$x_3$};
			\node[node_style,inner sep=0.01cm] (2) at (-0.5,-1) {$x_2$};
			\draw[-stealth,semithick] (1) -- (2);
			\draw[-stealth,semithick] (1) -- (3);
            \draw[-stealth,semithick,line join=round,
             decorate, decoration={
                zigzag,
                segment length=4,
                amplitude=.9,
                post=lineto,
                post length=5pt
            }] (u) -- (1);
            
            
			\node[text width=0.4cm] at (0,-2)  {\small(a)} ;
		\end{scope}

                        
            
   

		\begin{scope}[xshift=80]
			\draw[draw=red,line width=0.01,fill=BrickRed!10] (0.5,-1) ellipse (0.4 and 0.4);

			\node[node_style,inner sep=0.01cm,inner sep=0.01cm] (1) at (0,0) {$\hat{x}_1$};
			\node[node_style,inner sep=0.01cm,inner sep=0.01cm] (3) at (0.5,-1) {$\hat{x}_3$};
			\node[node_style,inner sep=0.01cm,inner sep=0.01cm] (2) at (-0.5,-1) {$\hat{x}_2$};
            \node[node_style,inner sep=0.01cm,,inner sep=0.01cm,fill=inputcolor] (u) at (0,1) {$u$};
            \draw[-stealth,semithick,line join=round,
            decorate, decoration={
                zigzag,
                segment length=4,
                amplitude=.9,post=lineto,
                post length=5pt
            }] (u) -- (1);
			\draw[-stealth,semithick] (1) -- (2);
			\draw[-stealth,semithick] (1) -- (3);
			\draw[-stealth,semithick] (3) to [in=210-180+50,out=160-180+50,loop,looseness=5] (3);

			\node[text width=0.4cm] at (0,-2)  {\small(b)} ;
		\end{scope}

		\begin{scope}[xshift=170]
			
			\draw[draw=red,line width=0.01,fill=BrickRed!10] (1.05,-0.65) ellipse (0.9 and 0.9);
			   
			\node[node_style,inner sep=0.01cm] (1) at (0,0) {$\hat{x}_1$};
			\node[node_style,inner sep=0.01cm] (3) at (0.55,-1) {$\hat{x}_3$};
			\node[node_style,inner sep=0.01cm] (2) at (-0.5,-1) {$\hat{x}_2$};
			\node[node_style,inner sep=0.01cm] (4) at (1,-0.1) {$\hat{x}_4$};
			\node[node_style,inner sep=0.01cm] (5) at (1.5,-1) {$\hat{x}_5$};
			\node[node_style,fill=inputcolor,inner sep=0.01cm] (u) at (0,1) {${u}$};
            \draw[-stealth,semithick,line join=round,
             decorate, decoration={
                zigzag,
                segment length=4,
                amplitude=.9,
                post=lineto,
                post length=5pt
            }] (u) -- (1);
			\draw[-stealth,semithick] (1) -- (3);
			\draw[-stealth,semithick] (3) -- (4);
            \draw[-stealth,semithick] (4) -- (3);
            
			\draw[-stealth,semithick] (4) -- (5);
            \draw[-stealth,semithick] (5) -- (4);
			\draw[-stealth,semithick] (5) -- (3);
            \draw[-stealth,semithick] (3) -- (5);
            
			\draw[-stealth,semithick] (3) -- (2);
			   
			\node[text width=0.4cm] at (0.5,-2) {\small(c)} ;

            \draw[-stealth,semithick] (3) to [out=210-70,in=160-70,loop,looseness=5] (3);
            \draw[-stealth,semithick] (4) to [out=210-50,in=160-50,loop,looseness=5] (4);
            \draw[-stealth,semithick] (5) to [in=210-180+70,out=160-180+70,loop,looseness=5] (5);
            
		\end{scope}

	\end{tikzpicture}
}}\vspace{-0.25cm}
    \caption{Example of original network (a). Node 3 of (a) is replaced by a subsystem with first-order heterogeneous dynamics (b). Node 3 of (a) is replaced by a subsystem with higher-order heterogeneous dynamics (c). The extended 
    subsystems are marked in red. All the edges have generic weights. The input node is highlighted in violet.}
    \label{fig:lift_ex}
    \vspace{-0.5cm}
\end{figure}
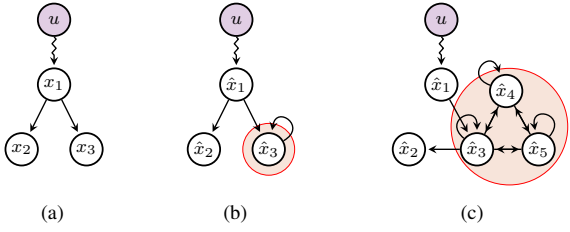

Figure~\ref{fig:lift_ex} illustrates two simple examples of construction of extended systems as in \eqref{eq:netsys-exp}.

In the following sections, we evaluate and quantify the advantages of employing heterogeneous and/or higher-order internal dynamics by comparing the structural output controllability of $\Sigma$ and $\hat{\Sigma}$. 
Notice that the structural output controllability of the former is equivalent to its structural (state) controllability. With our definitions, the dimension of the output for the extended system $\hat{\Sigma}$ matches the one of the original one $\Sigma$: in some sense, the output controllability of $\hat{\Sigma}$ is comparable to the controllability of ${\Sigma}$.

\begin{remark}
The construction of the extended system \eqref{eq:netsys-exp} shares similarities with the notion of lift of a graph  introduced for Markov chains in \cite{PD-SH:00,AST:21}, as both involve an expanded representation of the original network dynamics. However, a key distinction lies in how nodes without self-loops are treated. In the standard definition of lift of a graph, if a node in the original network does not have a self-loop, its lifted version cannot exhibit non-trivial internal dynamics, i.e., internal interconnections or self-loops.
\end{remark}

\section{Structural output controllability analysis of the extended system}\label{sec:output-controllability-modified}
\subsection{Graph-theoretic representations and structural controllability}\label{sec:graphlift}

Throughout the paper, we adopt a mixed graph-theoretic and algebraic approach to investigate structural output controllability of the extended system. To leverage existing graph-theoretic conditions for structural controllability, we need a suitable representation of the~original~and~modified~systems.

 In this subsection, we consider the system $\Sigma$ in \eqref{eq:netsys}, but all definitions and results apply to the extended system $\hat{\Sigma}$ in \eqref{eq:netsys-exp}, up to notational changes.

To $\Sigma$ we associate the directed graph ${\rm G}(\Sigma)={({\cal W},\Ec)}$, where ${\cal W}=\Vc\,\cup\, \Uc$. Here $\mathcal{V}$ is the set of state nodes, node $x_i\in \Vc$ is associated with the $i$-th entry of the state vector of $\Sigma$, $\Uc$ is the set of input nodes, and node $u_i\in\Uc$ is associated to the $i$-th component of the input of $\Sigma$. As edge set we consider ${\Ec}={\rm \Ec}_v\cup{\rm \Ec}_u$ where ${\Ec}_u= \{(u_i,x_j)\,|\, u_i\in \mathcal{U}, x_j\in {\mathcal{V}}, B_{ji}\neq 0 \}$ and ${\Ec}_v= \{(x_i,x_j)\,|\, x_i,x_j\in {\mathcal{V}}, A_{ji}\neq 0 \}$ and $A$, $B$ are the state, input matrices of $\Sigma$. It is possible to consider ${\rm G}({\Sigma})$ as composed by $n$ subgraphs describing the subsystems of $\Sigma$.
The $i$-th induced subgraph \cite{FB:24} is ${\rm G}_i(\Sigma)=(\Vc_i\cup {\Uc}, {{ \Ec}}_i)$, where the nodes in $\mathcal{{V}}_i$ are associated with the state of the $i$-th subsystem. Notice that $\Vc_i\cap \Vc_j=0$ and ${\Ec}_i\cap { \Ec}_j=0$, $\forall i\neq j$. 

To make the paper self-contained we recall necessary and sufficient conditions for structural controllability. 

 Some preliminary definitions  are in order. 
 {A {\em stem}  is a sequence of distinct nodes
$\{v_1,\dots,v_\ell\}\subseteq {\cal W}$ such that $(v_k,v_{k+1})\in{\Ec}$ for all $ k\in\{1,\dots,\ell-1\}$ with $v_1 = u_i\in{\cal U}$ and $v_\ell = w_j\in{\cal V}$.  
}An {\textit{elementary cycle} (or simply a cycle, with a slight abuse of terminology)  is a a sequence of distinct nodes
$\{v_1,\dots,v_\ell\}$ such that $(v_k,v_{k+1})\in{\Ec}$ for all $ k\in\{1,\dots,\ell-1\}$, and $v_1=v_\ell.$}

A set of nodes $\cal S$ is \textit{covered} (or spanned) by a subgraph ${\rm G}'(\Sigma)=({\cal W}',\Ec'$) of ${\rm G}(\Sigma)=({\cal W}, {\cal E})$ if ${\cal S}\subseteq$ ${\cal W}'$, and a graph is covered by a subgraph if the subgraph and the graph have the same vertex set. 

The main class of systems that will be considered in the following of the paper is that of \textit{input-accessible systems}.  A system is \textit{input-accessible} if every state vertex in $\mathrm{G}(\Sigma)$ is accessible, i.e., it is the end vertex of a stem. For these systems, the following theorem \cite{SH:1980,JMD-CC-JV:03} provides a graph-theoretic way to check the generic dimension of the controllable subspace and structural controllability.
\begin{theorem}\label{thm:gen_dim}
   The generic dimension of the controllable subspace of an input-accessible, structured system $\Sigma$ equals the maximum size of a set $\Zc \subseteq \mathcal{V}$ that can be covered by vertex-disjoint stems and elementary cycles in ${\rm G}(\Sigma)$. The system is structurally controllable if and only if $\Zc=\Vc$.
\end{theorem}

\subsection{Structural output controllability of extended systems}
In this subsection, employing the tools introduced from the structural controllability analysis, we present some results on the structural output controllability of the extended system \eqref{eq:netsys-exp} which will be instrumental in the subsequent analysis. 
We start with a preliminary lemma that restricts the class of systems which can benefit from introducing higher-order dynamics.
\begin{lemma}
    The extended system $\hat{\Sigma}$ in \eqref{eq:netsys-exp} is structurally output controllable only if the corresponding original network $\Sigma$ in \eqref{eq:netsys} is input-accessible.
\end{lemma}
\begin{proof}
     Let $\Vc_{{a}}$ be the set of accessible state vertices in ${\rm G}(\Sigma)$ and ${\Vc}_{\bar{a}}=\Vc\setminus\Vc_{{a}}$. 
     By contradiction, assume that $\Vc_{\bar{a}}$ is not empty. Without loss of generality, it is possible to choose a basis for the state space of $\Sigma$ such that ${x}^\top=[{x}_a^\top \ {x}_{\bar{a}}^\top]$ where ${x}_a, {x}_{\bar{a}}$ are the entries of $x$ associated to nodes in ${\Vc}_{{a}}$, ${\Vc}_{\bar{a}}$, 
     respectively. Then, ${a}_{ij}=0$, ${b}_{i}=0$ $\forall i,j$ such that $x_j\in\Vc_{{a}}$ and $x_i\in\Vc_{\bar{a}}$ to guarantee that the nodes $x_i\in\Vc_{\bar{a}}$ are not the end vertices of any stem (and therefore input-accessible). This implies $\hat{a}_{ij}=0$, $\hat{b}_{i}=0$ $\forall i,j$ such that $x_j\in\Vc_{{a}}$ and $x_i\in\Vc_{\bar{a}}$ for the constraints on the structure of $\hat{\Sigma}$. Partitioning the state vector of $\hat{\Sigma}$ accordingly to $x$, i.e.~$\hat{x}^\top=[\hat{x}_a^\top \ \hat{x}^\top_{\bar{a}}]$, the state, input and output matrices of the modified network have the following structure
    \begin{align*}
       \hat{A}=\begin{bmatrix}
           \hat{A}_a & *\\
           0 & \hat{A}_{\bar{a}}
       \end{bmatrix}, \quad \hat{B}=\begin{bmatrix}
            \hat{B}_a\\
            0
        \end{bmatrix}, \quad \hat{C}=\begin{bmatrix}
            \hat{C}_a & 0 \\
            0 & \hat{C}_{\bar{a}}
            \end{bmatrix}.
    \end{align*}
    The output controllability matrix is
    \begin{align*}
        \hat{\Rc}_o=\begin{bmatrix}
            \hat{C}_a\hat{B}_a & \hat{C}_a\hat{A}_a \hat{B}_a& \dots & \hat{C}_a\hat{A}^{\hat{n}-1}_a \hat{B}_a\\
            0 & 0 & \dots & 0
        \end{bmatrix}
    \end{align*}
    which is not of full row-rank independently of the choice of the parameters of the system. The system $\hat{\Sigma}$ is not structurally output controllable, which leads to the contradiction.
\end{proof}

We remark that input accessibility is not a restrictive assumption to our aim, as it is necessary also for the controllability of the original system.
For this reason, in the rest of the paper we focus on input-accessible network systems that are not structurally controllable. Regarding the output controllability of $\hat{\Sigma},$ the following basic characterizations hold.
\begin{lemma}\label{lem:str_out_ctrb} Consider the extended system $\hat{\Sigma}$ in \eqref{eq:netsys-exp}. The following hold:
    \begin{enumerate}
        \item If $\hat{\Sigma}$ is structurally controllable then it is structurally output controllable;
        \item $\hat{\Sigma}$ is  structurally output controllable only if $d_c(\hat{\Sigma})\geq n$, where $d_c(\hat{\Sigma})$ denotes the generic dimension of the controllability subspace of $\hat{\Sigma}$. 
    \end{enumerate}
\end{lemma}
\begin{proof}
We first prove 1). Let $\hat{\Rc}$ denote the controllability matrix of $\hat{\Sigma}$. Since $\hat{\Sigma}$ is structurally controllable, there exists a realization of $\hat{A}$, $\hat{B}$ leading to a full-rank controllability matrix $\hat{\Rc}$ and $\rank({{\hat{C}}\hat{\Rc}})=\rank({{\hat{C}}})$. The system is structurally output controllable since there exists a full row-rank realization of ${\hat{C}}$ by its definition. 
    
We now prove 2) by contradiction. Assume $\hat{\Sigma}$ is structurally output controllable and $d_c(\hat{\Sigma})<n$. Consider generic realizations of $\hat{A}$, $\hat{B}$, $\hat{C}$. Then $r_c=\rank(\hat{C})\leq n$ and $r_r=\rank(\hat{\Rc})\leq d_c< n$. This implies $\rank(\hat{C}\Rc)\leq \min\{r_c,r_r\} <n$ for every realization of $\hat{A}$, $\hat{B}$, $\hat{C}$. This implies that ${\hat{\Sigma}}$ is not structurally output controllable, leading to a contradiction.
\end{proof}

Finally, we recall a well-known result in structural controllability theory and adapt it to the extended system \eqref{eq:netsys-exp}. Input accessibility is sufficient for structural controllability, provided one allows replacing {\em all} the subsystems of $\Sigma$ with heterogeneous ones in $\hat{\Sigma}$.
For first-order systems, this is equivalent to add generic self-loops weights to the network graph.
The following proposition follows from classical results in structural controllability theory \cite{CC-CE:2012,LY-BA:2016}.
\begin{proposition}
    If $\Sigma$ is input-accessible, there always exists a structurally output controllable system $\hat{\Sigma}$ with heterogeneous first-order local dynamics.
\end{proposition}

The same can be accomplished if one employs subsystem of order strictly larger than 1, in which all $\hat a_{i}$ are generic.
However, as we shall see in the following sections, it is in general not necessary to replace all subsystems in the network with heterogeneous ones to achieve structural output controllability.

\section{Advantages of  higher-order local dynamics}\label{sec:advantages}

\subsection{How to measure the advantage}
In the following subsections, we shall show that introducing higher-order subsystems offers advantages over employing (heterogeneous) first-order ones for the following reasons:
\begin{enumerate}[{A}1)]
    \item \label{adv1} The subsystems described by \eqref{eq:subsys} lack heterogeneous internal dynamics. By adopting higher-order subsystems it is still possible to obtain a structurally output controllable $\hat{\Sigma}$ with homogeneous internal dynamics. 
    \item \label{adv2} In general, it is not necessary to modify all subsystems in the original network to achieve structural output controllability. The number of subsystems that require replacement with higher-order ones can be lower than the number of required heterogeneous~first-order~subsystems. 
\end{enumerate}
To quantify the advantage A\ref{adv2}) above, we consider the following index:
 \begin{equation}\label{eq:Delta}
        \Delta=S-\hat{S},
    \end{equation}
    where \(\hat{S}\) represents the number of subsystems where local higher-order dynamics is introduced to achieve structural output controllability of \(\hat{\Sigma}\), while \(S\) the minimum number of subsystems where first-order heterogeneous internal dynamics has to be introduced. The advantage of employing higher-order dynamics is verified when \(\Delta > 0\).
    
   \subsection{A characterization of systems that are not structurally controllable} \label{sec:charact}
    
   Our interest is in input-accessible network systems $\Sigma$ that are not structurally controllable. The following corollary of Theorem \ref{thm:gen_dim} offers a useful characterization that highlights what are the critical interconnections that need to be resolved.

   \begin{corollary}
    \label{cor:ctrb_cases}
    An input-accessible network system $\Sigma$ is not structurally controllable if and only if every set of stems and cycles covering the state nodes in ${\rm G}(\Sigma)$ includes one of the following:
    \begin{enumerate}
    \item A stem and a cycle intersecting at some node $w$;
    \item Two cycles intersecting at some node $w$;
    \item Two stems, intersecting at some node $w$, originated from different input nodes;
    \item Two stems originated from the same input node. 
\end{enumerate}
\end{corollary}
\begin{proof} 
If 1)-4) do not hold, the state nodes in ${\rm G}(\Sigma)$ can be covered by
vertex-disjoint stems and cycles. By Theorem \ref{thm:gen_dim} the system is structurally controllable.
If at least one of conditions 1)-4) holds, the same theorem ensures that the generic dimension of the controllable subspace is not equal to the system size, so that the system is not structurally controllable.
\end{proof}

Using the result above, we can define two classes of network systems, which require different types of higher-order dynamics to be made structurally controllable.

\begin{definition}
An input-accessible network system $\Sigma$  is called an:
\begin{itemize}
     \item {\bf X-network}, {if it is  not structurally controllable} and {\em there exists} a set of stems and cycles covering the state nodes in ${\rm G}(\Sigma)$ satisfying conditions 1)-2)-3) above, but not 4);
        \item {\bf Y-network}, if {\em every} set of stems and cycles covering the state nodes in ${\rm G}(\Sigma)$ satisfies condition 4) above, {but never  1)-2)-3)}.
\end{itemize} 
\end{definition}

The name of these classes is inspired by the shape of interconnection exhibited by the network that prevents structural controllability (examples are provided in Figure \ref{fig:XY_ex}(a)-(b)-(c) for X-networks, and Figure \ref{fig:XY_ex}(d) for Y-networks).
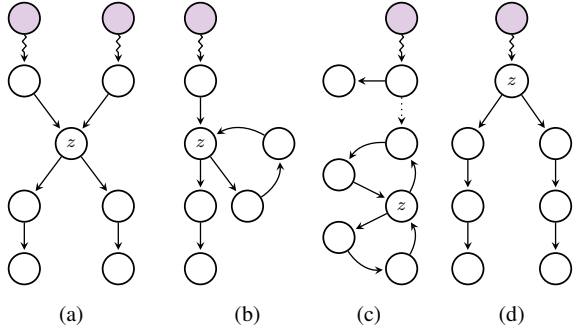
\begin{figure}
    \centering
    \resizebox{0.425\textwidth}{!}{

	\begin{tikzpicture}[shorten >=1pt, auto, ultra thick,
		node_style/.style={draw, circle,thick, fill=white, minimum size=0.5cm,font=\footnotesize},every edge/.append style = {thick}]
   \begin{scope}
   \node[node_style,inner sep=0.01cm] (1) at (0,0) {$z$};
   \node[node_style,inner sep=0.01cm] (3) at (0.75,-1) {};
   \node[node_style,inner sep=0.01cm] (2) at (-0.75,-1) {};
   \node[node_style,fill=inputcolor,inner sep=0.01cm] (u1) at (-0.75,2) {};
   \node[node_style,fill=inputcolor,inner sep=0.01cm] (u2) at (0.75,2) {};
   \node[node_style,inner sep=0.01cm] (ei1) at (-0.75,1) {};
   \node[node_style,inner sep=0.01cm] (ei2) at (0.75,1) {};
   
   \node[node_style,inner sep=0.01cm] (d1) at (-0.75,-2) {};
   \node[node_style,inner sep=0.01cm] (d2) at (0.75,-2) {};
   \draw[-stealth,semithick] (1) -- (2);

   \draw[-stealth,semithick,line join=round,
    semithick,
    decorate, decoration={
        zigzag,
        segment length=4,
        amplitude=.9,post=lineto,
        post length=5pt
    }] (u1) -- (ei1);

    \draw[-stealth,semithick,line join=round,
    semithick,
    decorate, decoration={
        zigzag,
        segment length=4,
        amplitude=.9,post=lineto,
        post length=5pt
    }] (u2) -- (ei2);

   \draw[-stealth,semithick,] (ei1) -- (1);
   \draw[-stealth,semithick,] (ei2) -- (1);
   \draw[-stealth,semithick,] (3) -- (d2);
   \draw[-stealth,semithick,] (2) -- (d1);
   \draw[-stealth,semithick,] (1) -- (3);
   
    \node at (0,-2.75) {(a)};
   
    \end{scope}

    \begin{scope}[xshift=80]
    \node[node_style,inner sep=0.01cm] (1) at (-0.75,0) {$z$};
   \node[node_style,inner sep=0.01cm] (3) at (0,-1) {};
   \node[node_style,inner sep=0.01cm] (4) at (0.5,0) {};
   \node[node_style,inner sep=0.01cm] (2) at (-0.75,-1) {};
   \node[node_style,fill=inputcolor,inner sep=0.01cm] (u) at (-0.75,2) {};
   \node[node_style,inner sep=0.01cm] (ei1) at (-0.75,1) {};
   \node[node_style,inner sep=0.01cm] (d1) at (-0.75,-2) {};
   
   \draw[-stealth,semithick] (1) -- (2);
   \draw[-stealth,semithick,line join=round,
    semithick,
    decorate, decoration={
        zigzag,
        segment length=4,
        amplitude=.9,post=lineto,
        post length=5pt
    }] (u) -- (ei1);
    
   \draw[-stealth,semithick] (2) -- (d1);
   \draw[-stealth,semithick] (ei1) -- (1);
   \draw[-stealth,semithick] (1) -- (3);
   \draw[-stealth,semithick] (3) to [bend right] (4);   
   \draw[-stealth,semithick] (4) to [bend right] (1);
   
    \node at (0,-2.75) {(b)};
    \end{scope}
   

    \begin{scope}[xshift=150,rotate=270]
         \node[node_style] (1) at (0-1,0) {};
    \node[node_style,inner sep=0.01cm] (3) at (1-1,0) {};
   \node[node_style,inner sep=0.01cm] (6) at (2-1,0) {$z$};
   \node[node_style,inner sep=0.01cm] (4) at (1.5-1,-1) {};
   \node[node_style,inner sep=0.01cm] (2) at (0-1,-1) {};
   \node[node_style,fill=inputcolor,inner sep=0.01cm] (u) at (-2,0) {};
    \node[node_style,inner sep=0.01cm] (5) at (2.5-1,-1) {};
    \node[node_style,inner sep=0.01cm] (7) at (3-1,0) {};
    
   \draw[-stealth,semithick] (1) -- (2);
  
   \draw[-stealth,semithick,line join=round,
    semithick,
    decorate, decoration={
        zigzag,
        segment length=4,
        amplitude=.9,post=lineto,
        post length=5pt
    }] (u) -- (1);

   \draw[-stealth,dotted,semithick] (1) to (3);
   \draw[-stealth,semithick] (3) to [bend right] (4);
   \draw[-stealth,semithick] (4) -- (6);
   \draw[-stealth,semithick] (6) to [bend right] (3);
   
   \draw[-stealth,semithick] (7) to [bend right] (6);
   \draw[-stealth,semithick] (6) -- (5);
   
   \draw[-stealth,semithick] (5) to [bend right] (7);

\end{scope}
    \node at (4.75,-2.75) {(c)};
    
\begin{scope}[xshift=200]   
\node[node_style,] (1) at (0,0+1) {$z$};
\node[node_style,] (2) at (-0.7,0) {\ };
\node[node_style,] (3) at (-0.7,-1) {\ };
\node[node_style,] (4) at (-0.7,-2) {\ };
\node[node_style,] (h2) at (0.7,-0) {\ };
\node[node_style,] (h3) at (0.7,-1) {\ };
\node[node_style,] (h4) at (0.7,-2) {\ };

\node[node_style,fill=inputcolor] (u) at (0,2) {};


\draw[-stealth,semithick,line join=round,
semithick,
decorate, decoration={
    zigzag,
    segment length=4,
    amplitude=.9,post=lineto,
    post length=5pt
}] (u) -- (1);

\draw[-stealth,semithick] (1) -- (2);
\draw[-stealth,semithick] (2) -- (3);
\draw[-stealth,semithick] (3) -- (4);
\draw[-stealth,semithick] (1) -- (h2);

\draw[-stealth,semithick] (h2) -- (h3);
\draw[-stealth,semithick] (h3) -- (h4);

\node at (0,-2.75) {(d)};
\end{scope}

    \end{tikzpicture}
  }
    \vspace{-0.75cm}
    \caption{{Examples of X-networks are given in subfigures (a)-(b)-(c), and Y-networks in (d). Each network is covered by a set of paths which exhibit a critical interconnection outlined in Corollary \ref{cor:ctrb_cases}. Two non-vertex-disjoint stems originated from different input nodes are depicted in (a), a stem sharing a node with a cycle in (b), two non-vertex-disjoint cycles in (c). (d) presents two stems that share a input and a state node. The input nodes are highlighted in violet. The edges connecting nodes in the considered covering paths are solid, other edges are dotted.}}
    \vspace{-0.5cm}
    \label{fig:XY_ex}
\end{figure}

 Notice that Y-networks  satisfy condition 4) above for all covering sets of stems and cycles, and hence they are (also) not structurally controllable.
{Furthermore, observe that every Y-network must be acyclic. This follows from input accessibility: if a cycle were present, it would necessarily be connected to the input, implying the existence of a set of intersecting stems and cycles covering state nodes.}

Given the definitions, controllability in X-networks can only be prevented by paths in the covering set of stems and cycles intersecting only for {\em state} nodes, whereas in Y-network it is essentially prevented by intersections in the {\em input} nodes. As we show in the next section, a modified network can be obtained  for X-networks by separating intersecting paths by introducing higher-order {\em homogeneous} dynamics. On the other hand, Y-networks cannot be modified in this way to obtain disjoint paths, due to the fact that we cannot modify the input nodes. In this case, the addition of heterogeneous higher-order dynamics is necessary to achieve structural output controllability. It is easy to see that X and Y networks do not comprise all the possible systems.  We will later discuss how  controllability can  be obtained with higher-order local dynamics also for systems that are neither type X nor Y.

\subsection{Making X-networks structurally  controllable}\label{sec:X-net}

We next show that X-networks can be made structurally controllable and, consequently, structurally output-controllable, by introducing homogeneous higher-order dynamics.  Specifically, we replace a set of subsystems in $\Sigma$ with higher-order subsystems in $\hat{\Sigma}$ with trivial internal dynamics. The dynamics of each modified subsystem and of the induced overall system $\hat{\Sigma}$ are described by \eqref{eq:subsys-exp} and \eqref{eq:netsys-exp} respectively, with $\hat{a}_i=0$, for all $i=1,\dots,n$. 

The following theorem shows that by carefully selecting the subsystems to be modified in $\hat{\Sigma}$  it is always possible to obtain a structurally controllable system $\hat{\Sigma}$ with homogeneous internal dynamics.
\begin{theorem}
\label{thm:hoi}
Consider a X-network ${\Sigma}$. Then there exists a structurally controllable system~$\hat{\Sigma}$ composed of subsystems with homogeneous internal dynamics.
\end{theorem}
\begin{proof}
Since $\Sigma$ is a X-network there exists a set of stems and cycles covering the state nodes in ${\rm G}(\Sigma)$ which satisfies conditions 1)-2)-3) in Corollary \ref{cor:ctrb_cases}, but not 4). Let $\Pc = \{p_1, \dots, p_\ell\}$ be one such sets, and denote by $\mathcal{Z}$ the set of state nodes shared by at least two distinct paths in $\Pc$.
In the following, we give a constructive graph-theoretic method to obtain a structurally controllable extended network $\hat{\Sigma}$ with homogeneous internal dynamics.

We start by replacing each subsystem in $\Sigma$ corresponding to a shared node $z\in\mathcal{Z}$ with an higher-order subsystem in $\hat{\Sigma}$ whose order $\ell$ is equal to the number of paths in $\Pc$ sharing $z$. 
 The addition of local state variables is equivalent, using the notation introduced in Section \ref{sec:graphlift}, to an extension of the original graph $\mathrm{G}({\Sigma})$ to $\mathrm{G}(\hat{\Sigma})$ by the addition of $\ell-1$ nodes for each node $z\in\mathcal{Z}$. So each $z$ now corresponds a subgraph, which we denote by $\mathrm{G}_z(\hat{\Sigma}),$ of $\mathrm{G}(\hat{\Sigma}).$  We next discuss how these new nodes need to be connected to the rest of the graph in order to attain controllability. 

If $z\in\mathcal{Z}$ and $z$ is shared between $q$ distinct paths $\{p_{k_i}\}_{i=1}^q\subseteq \mathcal{P}$,  it will be convenient to label the nodes in the subgraph $\mathrm{G}_z(\hat{\Sigma})$ as $z^{k_1},\dots,z^{k_q}$.
 If $z\not\in\mathcal{Z}$, it corresponds to a subsystem that has not been replaced, and we maintain for it the same label $z$ also in ${\rm G}(\hat\Sigma).$ Notice we are  considering also fictitious subsystems associated to input nodes, which however will not need to be replaced since the system is a X-network.

We next construct a set $\hat{\Pc}=\{\hat{p}_1, \dots, \hat{p}_\ell\}$ of sequences of nodes in ${\rm G}(\hat{\Sigma})$ that will take the role of $\Pc.$  Let $\hat{p}_{k,s},p_{k,s}$ be the $s$-th elements in the sequences $\hat{p}_k$, ${p}_k$ respectively. For all $k,s,$ we set:
\begin{enumerate}
    \item $\hat{p}_{k,s}=p_{k,s}$ if $p_{k,s}\notin \Zc$; 
    \item $\hat{p}_{k,s}=p_{k,s}^{k}$ if $p_{k,s}\in \Zc$. 
\end{enumerate}
 We next add to ${\rm G}(\hat{\Sigma})$ the edges connecting each of the sequences $\hat p_k.$ 
$\hat{\Pc}$ is now a set of paths in ${\rm G}(\hat{\Sigma})$ and, by construction, it covers all state nodes. 
Notice that if $p_k$ is a stem or a cycle then $\hat{p}_k$ is also a stem or cycle, respectively. Moreover all paths $\hat{p}_k$ are disjoint since:
(i) the state nodes in $\hat{p}_i$ are different from the state nodes in $\hat{p}_j$ for all $i,j$ with $i\ne j$; and (ii) if $\hat{p}_i$ and $\hat{p}_j$ are different stems then they originate from different input nodes since $\Sigma$ is a X-network. 

Finally, to ensure input-accessibility, we define the full edge set $\hat {\cal E}$ of ${\rm G}(\hat{\Sigma})$ as follows: 
\begin{enumerate}
    \item $(i,j)\in\hat{\Ec}$ if $(i,j)\in \Ec$, $i,j\notin \mathcal{Z}$;
    \item $(i,j^k)\in\hat{\Ec}$ if $(i,j)\in \Ec$, $i\notin \Zc,j\in {\Zc},\forall k$;
    \item $(i^k,j)\in\hat{\Ec}$ if $(i,j)\in \Ec$, $i\in \Zc,j\notin {\Zc},\forall k$;
    \item $(i^k,j^s)\in\hat{\Ec}$ if $(i,j)\in \Ec$, $i,j\in \Zc,\forall k,s$.
\end{enumerate}
 Note that the above edge choice also includes (as a subset) the edges connecting the nodes $\hat p_k$ we had already added. 
 In addition, the rules 1)-2)-3)-4) above imply that nodes associated to the $i$-th subsystem in ${\rm G}(\hat{\Sigma})$ are connected to nodes associated to the $j$-th subsystem if and only if nodes associated to the $i$-th subsystem were already connected to nodes associated to the $j$-th one in $\mathrm{G}(\Sigma)$. 
{Note that in the state matrix of the extended system $\hat\Sigma$ we choose $\hat{a}_{ij}\neq 0$ only if there is an edge from a {state} node in $\mathrm{G}_j(\hat{\Sigma})$ to a {state} node in $\mathrm{G}_i(\hat{\Sigma})$.} Similarly in the input matrix of $\hat{\Sigma}$ we choose $\hat{b}_{k,s} \neq 0$ only if in $\mathrm{G}_k(\hat{\Sigma})$ there is an edge from the s-th input node to one of the state nodes.
Hence, by construction, the state and input matrices of the extended system are such that $\hat{a}_{ij}=0$, {$\hat{b}_{k,s}=0$}   only if  ${a}_{ij}=0$, {${b}_{k,s}=0$}, and therefore the locality constraints on the interconnections specified in Sec.~\ref{sec:hod} are satisfied.
The above construction yields an input-accessible system $\hat{\Sigma},$ with the state nodes in $\mathrm{G}(\hat{\Sigma})$ that can be covered by vertex-disjoint stems and cycles. By Theorem \ref{thm:gen_dim}, $\hat{\Sigma}$ is structurally controllable. Note that all subsystems in $\hat{\Sigma}$ have homogeneous internal dynamics since if we consider the $j$-th subsystem the nodes $j^k$,  $\forall k$,~are~not~connected~each~other.
\end{proof}

The following corollary of Theorem \ref{thm:hoi} follows by directly applying Lemma~\ref{lem:str_out_ctrb}.
\begin{corollary}
The system $\hat{\Sigma}$ in the statement of Theorem 2 is structurally output controllable.
\end{corollary}
\begin{remark} \label{rem:nt-dyn} Notice that Theorem \ref{thm:hoi} offers a simple, constructive, and systematic method to address all critical interconnections of type 1)-2)-3) and make the extended network system structurally controllable. The critical intersection nodes are substituted with higher-order ones that have {\em trivial internal dynamics}: the higher-order systems are just integrators. It is only thanks to the extra dimensions (i.e. nodes) that new spanning graphs allow for interconnections that avoid the class X properties.{  Fig.~\ref{fig:hoi_only} shows examples of elementary X-networks that are rendered structurally controllable by introducing homogeneous higher-order subsystems, following the method described in the proof of Theorem 2.} \myend
\end{remark}

\begin{figure}[!h]
    \centering
    {\resizebox{0.40\textwidth}{!}{

	\begin{tikzpicture}[shorten >=1pt, auto, ultra thick,
		node_style/.style={draw, circle,thick, fill=white, minimum size=0.5cm,font=\footnotesize},every edge/.append style = {thick}]

     \begin{scope}[]
     \draw[draw=red,line width=0.01,fill=BrickRed!10] (0,0) ellipse (1.2 and 0.5);
    \node[node_style,inner sep=0.01cm] (1) at (-0.75,0) {$z^1$};
    \node[node_style,inner sep=0.01cm] (4) at (0.75,0) {$z^2$};
   \node[node_style,inner sep=0.01cm] (3) at (0.75,-1) { \ };
   \node[node_style,inner sep=0.01cm] (2) at (-0.75,-1) { \ };
   \node[node_style,fill=inputcolor,inner sep=0.01cm] (u1) at (-0.75,2) {};
   \node[node_style,fill=inputcolor,inner sep=0.01cm] (u2) at (0.75,2) {};
   \node[node_style,inner sep=0.01cm] (ei1) at (-0.75,1) { \ };
   \node[node_style,inner sep=0.01cm] (ei2) at (0.75,1) { \ };
   \node[node_style,inner sep=0.01cm] (d1) at (-0.75,-2) {};
   \node[node_style,inner sep=0.01cm] (d2) at (0.75,-2) {};
   \draw[-stealth,semithick] (1) -- (2);
    \draw[-stealth,dotted,semithick] (4) -- (2);

   \draw[-stealth,semithick,line join=round,
    semithick,
    decorate, decoration={
        zigzag,
        segment length=4,
        amplitude=.9,post=lineto,
        post length=5pt
    }] (u2) -- (ei2);

    \draw[-stealth,semithick,line join=round,
    semithick,
    decorate, decoration={
        zigzag,
        segment length=4,
        amplitude=.9,post=lineto,
        post length=5pt
    }] (u1) -- (ei1);

   \draw[-stealth,semithick] (ei1) -- (1);
   \draw[-stealth,dotted,semithick] (ei2) -- (1);
   \draw[-stealth,dotted,semithick] (ei1) -- (4);
   
   \draw[-stealth,semithick] (ei2) -- (4);
   \draw[-stealth,semithick] (4) --  (3);
   \draw[-stealth, dotted,semithick] (1) -- (3);
   
   \draw[-stealth,semithick] (2) -- (d1);
   \draw[-stealth,semithick] (3) -- (d2);
   
       \node[text width=0.4cm] at (0,-2.85)  {(a)} ;
    \end{scope}
    
    \begin{scope}[xshift=80]
         \draw[draw=red,line width=0.01,fill=BrickRed!10] (-0.12,0) ellipse (1 and 0.5);

   \node[node_style,inner sep=0.01cm] (1) at (-0.75,0) {$z^1$};
   \node[node_style,inner sep=0.01cm] (3) at (1.,-1) {};
   \node[node_style,inner sep=0.01cm] (4) at (1.5,0) {};
   \node[node_style,inner sep=0.01cm] (2) at (-0.75,-1) {};
   \node[node_style,fill=inputcolor,inner sep=0.01cm] (u) at (-0.75,2) {};
   \node[node_style,inner sep=0.01cm] (ei1) at (-0.75,1) {};
   \node[node_style,inner sep=0.01cm] (5) at (0.5,0) {$z^2$};
   \node[node_style,inner sep=0.01cm] (d1) at (-0.75,-2) {};
   
   \draw[-stealth,semithick] (1) -- (2);
   \draw[-stealth,semithick,dotted] (5) -- (2);
   \draw[-stealth,semithick,line join=round,
    semithick,
    decorate, decoration={
        zigzag,
        segment length=4,
        amplitude=.9,post=lineto,
        post length=5pt
    }] (u) -- (ei1);

   \draw[-stealth,semithick] (ei1) -- (1);
   \draw[-stealth,semithick,dotted] (ei1) -- (5);
   \draw[-stealth,semithick] (5) to [bend right] (3);
   \draw[-stealth,semithick,dotted] (1) -- (3);
   \draw[-stealth,semithick] (3) to [bend right] (4);

   \draw[-stealth,semithick] (4) to [bend right] (5);
   \draw[-stealth,semithick,dotted] (4) to [bend right] (1);
   
   \draw[-stealth,semithick] (2) -- (d1);   
   
   
    \node[text width=0.4cm] at (0.25,-2.85)  {(b)} ;
    \end{scope}



    \begin{scope}[xshift=180, rotate=270]
     \draw[draw=red,line width=0.01,fill=BrickRed!10] (1,-0.5) ellipse (0.5 and 0.9);
    
    \node[node_style,inner sep=0.01cm] (1) at (-1,0) {};
    \node[node_style,inner sep=0.01cm] (3) at (0,0) {\ };
   \node[node_style,inner sep=0.01cm] (6) at (1,0) {$z^1$};
   \node[node_style,inner sep=0.01cm] (8) at (1,-1) {$z^2$};
   \node[node_style,inner sep=0.01cm] (4) at (0,-1) { \ };
   \node[node_style,inner sep=0.01cm] (2) at (-1,-1) { \ };

   \node[node_style,fill=inputcolor,inner sep=0.01cm] (u) at (-2,0) {};
    \node[node_style,inner sep=0.01cm] (5) at (2,-1) { \ };
    \node[node_style,inner sep=0.01cm] (7) at (2,0) { \ };
   \draw[-stealth,semithick] (1) -- (2);
   \draw[-stealth,semithick,line join=round,
    semithick,
    decorate, decoration={
        zigzag,
        segment length=4,
        amplitude=.9,post=lineto,
        post length=5pt
    }] (u) -- (1);

   \draw[-stealth,semithick,dotted] (1) -- (3);
   \draw[-stealth,semithick] (3) to [bend right]  (4);
   \draw[-stealth,semithick] (4) -- (6);
   \draw[-stealth,semithick] (6) to [bend right] (3);
   \draw[-stealth,semithick,dotted] (8) -- (3);

   \draw[-stealth,semithick] (7) -- (8);
   \draw[-stealth,semithick,dotted] (7) -- (6);
   
   \draw[-stealth,semithick] (8) to [bend right] (5);
   \draw[-stealth,semithick,dotted] (6) -- (5);
   
   \draw[-stealth,semithick,dotted] (4) to [bend right] 
 (8);
   \draw[-stealth,semithick] (5) to [bend right] (7);

    \end{scope}
       \node[text width=0.4cm] at (5.75,-2.85)  {(c)} ;
    \end{tikzpicture}
  }}
    \caption{{Examples of higher-order subsystems placement and design described in the proof of Theorem \ref{thm:hoi}. The depicted networks  are the expanded  X-networks corresponding to those illustrated  in subfigures \ref{fig:XY_ex} (a)-b)-(c). The input nodes are highlighted in violet. Nodes corresponding to the same higher-order subsystem are marked in red. The edges connecting nodes in vertex-disjoint stems and cycles covering all state nodes are solid. The other edges are dotted.}}
    \vspace{-0.25cm}
     \label{fig:hoi_only}
\end{figure}
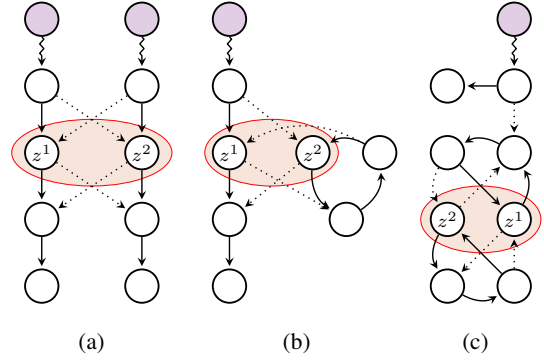

We next focus on assessing the advantage in the number of subsystems with respect to first-order dynamics. In the following lemma, we characterize the minimum number $S$ of first-order heterogeneous subsystems that need to be introduced to obtain a structurally output controllable system $\hat{\Sigma}$.
\begin{lemma}\label{lem:min_fo}
    Let $\Zc \subseteq \mathcal{V}$ be a maximum cardinality set of nodes that can be covered by vertex-disjoint stems and cycles in ${\rm G}(\Sigma)$, then the minimum number of subsystems that need to be replaced with first-order heterogeneous ones to attain structural output controllability is $S=n-|\Zc|$. 
\end{lemma}
\begin{proof}
Let us endow the $i$-th subsystem in $\hat{\Sigma}$ with first-order heterogeneous internal dynamics if $x_i\not\in \Zc$. ${\rm G}(\hat{\Sigma})$ has the same topology of ${\rm G}({\Sigma})$  except for additional self-loops at nodes $x_j$ satisfying $x_j\notin \Zc$. ${\rm G}(\hat{\Sigma})$ is therefore covered by vertex-disjoint stem and cycles and, by Theorem \ref{thm:gen_dim} and Lemma \ref{lem:str_out_ctrb} $\hat{\Sigma}$ is structurally output controllable. We prove minimality by contradiction, suppose we add first-order heterogeneous internal dynamics to ${S}<n-|\Zc|$ nodes and the system is structurally output controllable. By Theorem \ref{thm:gen_dim}, the generic dimension of the controllability subspace of $\hat{\Sigma}$, $d_c(\hat{\Sigma})$, satisfies {$d_c(\hat{\Sigma})\leq |\Zc|+S< n $}
and by Lemma \ref{lem:str_out_ctrb} the system is not structurally output controllable, leading to a~contradiction.~
\end{proof}

The following corollary of Theorem \ref{thm:hoi} shows that the benefit of employing higher-order dynamics over heterogeneous first-order ones, as quantified by $\Delta$ in \eqref{eq:Delta}, can grow unbounded with the network size for the case of X-networks.

\begin{corollary}
Let $\Delta$ be as in \eqref{eq:Delta}. There exist X-networks $\Sigma$ for which $\Delta \to \infty$ as $n \to \infty$.
\end{corollary}

\begin{proof}
Consider an X-network system $\Sigma$ where ${\rm G}(\Sigma)$ consists of a stem  $e_1$ and a cycle $e_2$, intersecting at a single node $z$. Following the same reasoning as in the proof of Theorem \ref{thm:hoi}, replacing the subsystem associated with $z$ with a second-order homogeneous subsystem results in a new system $\hat{\Sigma}$, where $\mathrm{G}(\hat{\Sigma})$ is composed of vertex-disjoint stems and cycles. By Theorem \ref{thm:gen_dim} and Lemma \ref{lem:str_out_ctrb}, $\hat{\Sigma}$ is structurally output controllable, implying $\hat{S} = 1$. In contrast, Lemma \ref{lem:min_fo} ensures that as the number of nodes in $e_1$ and $e_2$ grows, $S \to \infty$, leading to $\Delta \to \infty$.
\end{proof}

\subsection{Making Y-networks structurally output controllable}\label{sec:Y-net}

In this section, we {discuss how} introducing heterogeneous local dynamics is necessary to obtain structural output controllability in Y-networks. In addition, we demonstrate that higher-order heterogeneity provides an advantage over first-order heterogeneity, as measured by the~proposed~index~$\Delta$. 
\begin{proposition}\label{prop:not_str_out_ctrb}
Consider a Y-network $\Sigma$. Any extended $\hat{\Sigma}$ composed of subsystems with {\em homogeneous} internal dynamics is {not} structurally output controllable.
\end{proposition}
{
\begin{proof}
    As outlined in Section \ref{sec:charact}, the Y-network $\Sigma$ is acyclic and not structurally controllable. Consequently, by Theorem \ref{thm:gen_dim}, {the generic dimension of the controllable subspace of $\Sigma$,} $d_c({\Sigma}),$ is equal to the maximum size of a set $\Zc \subseteq \Vc$ that can be covered by vertex-disjoint stems in ${\rm G}(\Sigma)$ and $|\Zc| < n$. {Being disjoint, the number of these stems must be at most equal to the numbers of inputs, and for each input a stem of maximum length must be considered.}
    
    {Proving the statement, given genericity of controllability, is equivalent to show that a structured extended system $\hat\Sigma,$ with the constraints of having homogeneous (trivial) internal dynamics and satisfying the locality bounds, is not structurally controllable.} The constraints defined in Section \ref{sec:hod} imply that  $\hat{a}_{ij} \neq 0$ only if $a_{ij} \neq 0$ (generically), and hence ensure that state nodes in ${\rm G}_i(\hat{\Sigma})$ and ${\rm G}_j(\hat{\Sigma})$ are connected only if state nodes in ${\rm G}_i({\Sigma})$ and ${\rm G}_j({\Sigma})$ 
    are also (generically) connected. Similarly (since $\hat{b}_{i,s} \neq 0$ only if ${b}_{i,s} \neq 0$), the state nodes in ${\rm G}_i(\hat{\Sigma})$ are connected to the $s$-th input node only if the state node in ${\rm G}_i({\Sigma})$ is connected to the  $s$-th input node. {Moreover, since the internal dynamics of each extended subsystem is homogeneous (trivial), there are no edges between nodes within ${\rm G}_i(\hat{\Sigma})$.} {Therefore, the constraints on the interconnections and the homogeneity of internal dynamics imply that ${\rm G}(\hat{\Sigma})$ is also acyclic.} {Hence, as argued above for ${\rm G}({\Sigma}),$ the generic dimension of {the} controllability subspace of {$\hat{\Sigma}$} is equivalent to the largest number of nodes that can be covered by {\em disjoint stems} in {${\rm G}(\hat{\Sigma})$}.}  
    
    Since, by definition of Y-network, condition 3 in Corollary 1 is never met for ${\rm G}({\Sigma})$, all stems originating from different input nodes are vertex-disjoint, meaning {the state node $x_i$ in ${\rm G}_i(\Sigma)$ is only part of stems originating from the same input node}.
    {This, {with the imposed constraints on interconnections}, implies that all stems covering  state nodes in ${\rm G}_i(\hat{\Sigma})$ are originated from the same input node.}
    {Furthermore, a stem in ${\rm G}(\hat{\Sigma})$ can have state nodes belonging to ${\rm G}_{i_1}(\hat{\Sigma}), \dots, {\rm G}_{i_k}(\hat{\Sigma})$ only if there exists a stem in ${\rm G}(\Sigma)$ whose state nodes belong to ${\rm G}_{i_1}(\Sigma), \dots, {\rm G}_{i_k}(\Sigma)$. It then follows that two stems in ${\rm G}(\hat{\Sigma})$ are vertex-disjoint if and only if their counterparts in ${\rm G}(\Sigma)$ are originated by different input nodes.}
    
    {Also, observe that}, given acyclicity and homogeneity of internal dynamics, each stem in ${\rm G}(\hat{\Sigma})$ covers at most one node per ${\rm G}_{i}(\hat{\Sigma})$.  Consequently, for every stem in ${\rm G}(\hat{\Sigma})$ there exists a stem in ${\rm G}(\Sigma)$ of the same length. 
    
    Thus, by Theorem \ref{thm:gen_dim}, $d_c(\hat{\Sigma})$ is determined by the nodes covered by a set of vertex-disjoint stems, which, as we proved above, are at most one for each input, and of length smaller or equal to a corresponding one in the original graph. Hence, the total number of nodes covered by disjoint stems in the extended graph cannot be larger than that of the original graph, which corresponds to $d_c({\Sigma})<n$. Hence, by Lemma \ref{lem:str_out_ctrb}, $\hat{\Sigma}$ cannot be structurally output controllable.
\end{proof}
}

In view of the previous proposition, we consider replacing a subset of systems in $\Sigma$ with higher-order subsystems in $\hat{\Sigma}$ with \textit{heterogeneous internal dynamics}. The dynamics of each modified subsystem
is described by \eqref{eq:subsys-exp} where $\hat{a}_i$ is a generic matrix, the induced overall system $\hat{\Sigma}$ is described by \eqref{eq:netsys-exp}.

The following theorem gives bounds on the minimum number of higher-order subsystems that need to be introduced for the considered class of networks.
\begin{theorem}
\label{prop:bounds}
Consider a Y-network $\Sigma$. Let $\hat{n}_M={\rm max}_i \ \hat{n}_i$
and $\Zc \subseteq{\mathcal{V}}$ be a maximum cardinality set of nodes covered by vertex-disjoint stems 
in ${\rm G}(\Sigma).$ 
The minimum number $\hat{S}_{\min}$ of subsystems to be turned into higher-order heterogeneous ones in ${\rm G}(\hat\Sigma)$ to obtain structural output controllability satisfies
        \begin{align}\label{eq:bounds}
            \frac{n-|\Zc|}{\hat{n}_M} \le \hat{S}_{\min} \le n-|\Zc|.
            \end{align}
\end{theorem} 
  \begin{proof}
First notice that all the nodes associated with higher-order subsystems belong to a set of stems and cycles $\hat{\Pc}$ covering a maximum number of state nodes in ${\rm G}(\hat{\Sigma})$. Indeed,  since $\Sigma$ is a Y-network, and therefore acyclic, the constraints on the interconnections of $\hat{\Sigma}$ in Section \ref{sec:hod} impose that the only cycles in the network are composed by nodes in heterogeneous higher-order subsystems.
Moreover, each node associated with a heterogeneous higher-order subsystem has a self-loop. Therefore, for each cycle in the network there exists a set of self-loops covering the same nodes. This implies that a set $\hat{\Pc}$ of vertex-disjoint stem and cycles, which covers the maximum number of state nodes, is obtained by considering vertex-disjoint stems of maximum length in ${\rm G}(\hat{\Sigma})$ and self-loops corresponding to nodes not already considered in a stem. Therefore $\hat{\Pc}$ covers all nodes corresponding to higher-order subsystems.

Let $\Zc_a$ and $\Zc_c$ be, respectively, the sets of subsystems of $\hat{\Sigma}$ with state nodes in ${\rm G}(\hat{\Sigma})$ belonging to stems and having self-loops in $\hat{\Pc}$. Notice that $\Zc_a$, $\Zc_c$ are, in general, not disjoint as some nodes associated with a higher-order subsystem may belong to a stem in $\hat{\Pc}$, while others may have a self-loop in $\hat{\Pc}$. Moreover, first-order subsystems may belong only to $\Zc_a$. By Theorem \ref{thm:gen_dim}, the generic dimension of the controllable subspace is then given by the number of state nodes covered by $\hat{\Pc}$:
  \begin{align}\label{eq:lower-bound_new}
        d_c(\hat{\Sigma})= |\Zc_a|\! +\!\sum_{i\in \Zc_a} (\hat{n}_i-1) + \!\!\sum_{i\in \Zc_c {\setminus{\Zc_a}}}\!\!\hat{n}_i \leq |\Zc_a|\!+\!\hat{n}_M \hat{S}
  \end{align}
  where $\hat{S}$ is the number of higher-order subsystems in $\hat{\Sigma}$.
    
    The constraints on the interconnections impose a stem in ${\rm G}(\hat{\Sigma})$ can have state nodes belonging to ${\rm G}_{i_1}(\hat{\Sigma}), \dots, {\rm G}_{i_k}(\hat{\Sigma})$ only if there exists a stem in ${\rm G}(\Sigma)$ whose state nodes belong to ${\rm G}_{i_1}(\Sigma), \dots, {\rm G}_{i_k}(\Sigma)$. {Moreover, since $\Sigma$ does not satisfy condition 3 in Corollary 1, the constraints on interconnections implies stems in ${\rm G}(\hat{\Sigma})$ are vertex-disjoint only if their counterparts in ${\rm G}(\Sigma)$ are.} This implies $|\Zc_a|\leq |\Zc|$. 
    For $\hat{\Sigma}$ to be structurally output controllable, it must satisfy $d_c(\hat{\Sigma}) \geq n$ (Lemma \ref{lem:str_out_ctrb}). Combining this condition with \eqref{eq:lower-bound_new}, we obtain 
    \begin{align*}
        \hat{S} \geq \frac{n - |\Zc|}{\hat{n}_M},
    \end{align*}
   which leads to the lower bound in \eqref{eq:bounds}.
   
   If we replace subsystems corresponding to nodes {\textit{not}} in ${\Zc}$ with higher-order heterogeneous subsystems, ${\rm G}(\hat{\Sigma})$ is covered by vertex-disjoint stems and cycles. Therefore, by Theorem \ref{thm:gen_dim}, $\hat{\Sigma}$ is structurally controllable and by Lemma \ref{lem:str_out_ctrb} it is also structurally output controllable. In total, we have added $\hat{S}=n-|\Zc|$ higher-order subsystems leading to~the~upper~bound~in~\eqref{eq:bounds}.
     \end{proof}
     
Note that, in Theorem \ref{prop:bounds}, $n-|\Zc|$ equals the minimum number of first-order heterogeneous subsystems required to achieve structural (output) controllability. Therefore, the lower bound in the same theorem suggests that structural output controllability could be achieved with fewer higher-order subsystems than first-order heterogeneous ones, leading to an advantage in terms of the index $\Delta$ in \eqref{eq:Delta}. In the following, we show that this advantage is achievable for two classes of Y-networks. To this end, we will exploit some instrumental results that complement those in \cite{BD-JL-MJ:23} to assess output controllability using PBH-like tests. These results are presented and discussed~in~Appendix~\ref{sec:out_PBH}.
\subsubsection{Case study I. Binary tree network}
We consider the structured system described by the binary tree network ${\rm G}(\Sigma)$ with height $h$ of Fig.~\ref{fig:tree}(a). 
\begin{figure}[!h]
\begin{center}
\resizebox{0.41\textwidth}{!}{
\begin{tikzpicture}[node_style/.style={draw, circle,thick, fill=white, minimum size=0.435cm,font=\footnotesize}]
\begin{scope}
\node[node_style,inner sep=0.01cm] (1) at (0,0) {$1$};
\node[node_style,inner sep=0.01cm] (2) at (-0.7,-0.75) {$2$};
\node[node_style,inner sep=0.01cm] (3) at (0.7,-0.75) {$3$};
\node[node_style,inner sep=0.01cm] (4) at (-1.3,-1.5) {$4$};
\node[node_style,inner sep=0.01cm] (5) at (-0.3,-1.5) {$5$};
\node[node_style,inner sep=0.01cm] (6) at (0.3,-1.5) {$6$};
\node[node_style,inner sep=0.01cm] (7) at (1.3,-1.5) {$7$};
\node[node_style,inner sep=0.01cm]  at (-1.55,-2.5) {};
\node[node_style,inner sep=0.01cm]  at (-1.05,-2.5) {};
\node[node_style,inner sep=0.01cm]  at (0.5,-2.5) {};
\node[node_style,inner sep=0.01cm]  at (1.05,-2.5) {};
\node[node_style,inner sep=0.01cm]  at (1.55,-2.5) {};
\node  at (-0.25,-2.5) {\dots};
\node[node_style,inner sep=0.01cm,fill=inputcolor] (u) at (0,1) {$u$};
\node[inner sep=0cm] at (2.5,-1.25) {\ };

\draw[-stealth,semithick,line join=round,
semithick,
decorate, decoration={
    zigzag,
    segment length=4,
    amplitude=.9,post=lineto,
    post length=5pt
}] (u) -- (1);
\draw[-stealth,semithick] (1) -- (2);
\draw[-stealth,semithick] (1) -- (3);
\draw[-stealth,semithick] (2) -- (4);
\draw[-stealth,semithick] (2) -- (5);
\draw[-stealth,semithick] (3) -- (6);
\draw[-stealth,semithick] (3) -- (7);
\draw[dashed,semithick] (4) -- (-1.5,-2.25);
\draw[dashed,semithick] (4) -- (-1.1,-2.25);
\draw[dashed,semithick] (5) -- (-0.5,-2.25);
\draw[dashed,semithick] (5) -- (-0.1,-2.25);
\draw[dashed,semithick] (6) -- (0.5,-2.25);
\draw[dashed,semithick] (6) -- (0.1,-2.25);
\draw[dashed,semithick] (7) -- (1.5,-2.25);
\draw[dashed,semithick] (7) -- (1.1,-2.25);

\node at (-3.5,0.75) {(a)};
\end{scope}

\begin{scope}[yshift=-4.25cm]
    
\draw[draw=red,line width=0.01,fill=BrickRed!10] (0,0) ellipse (1 and 0.325);
\node[node_style,inner sep=0.01cm] (1) at (-0.4,0) {$1$};
\node[node_style,inner sep=0.01cm] (1b) at (0.4,0) {$2$};
\draw[draw=red,line width=0.01,fill=BrickRed!10] (-1,-0.75) ellipse (1 and 0.325);
\node[node_style,inner sep=0.01cm] (2) at (-1.4,-0.75) {$3$};
\node[node_style,inner sep=0.01cm] (2b) at (-0.6,-0.75) {$4$};
\draw[draw=red,line width=0.01,fill=BrickRed!10] (1.1,-0.75) ellipse (1 and 0.325);
\node[node_style,inner sep=0.01cm] (3) at (0.7,-0.75) {$5$};
\node[node_style,inner sep=0.01cm] (3b) at (1.5,-0.75) {$6$};
\draw[draw=red,line width=0.01,fill=BrickRed!10] (-3.1,-1.5) ellipse (1 and 0.325);
\node[node_style,inner sep=0.01cm] (4) at (-3.5,-1.5) {$7$};
\node[node_style,inner sep=0.01cm] (4b) at (-2.7,-1.5) {$8$};
\draw[draw=red,line width=0.01,fill=BrickRed!10] (-1.05,-1.5) ellipse (1 and 0.325);
\node[node_style,inner sep=0.01cm] (5) at (-1.45,-1.5) {$9$};
\node[node_style,inner sep=0.01cm] (5b) at (-0.65,-1.5) {$10$};
\draw[draw=red,line width=0.01,fill=BrickRed!10] (1.05,-1.5) ellipse (1 and 0.325);
\node[node_style,inner sep=0.01cm] (6) at (0.65,-1.5) {$11$};
\node[node_style,inner sep=0.01cm] (6b) at (1.45,-1.5) {$12$};
\draw[draw=red,line width=0.01,fill=BrickRed!10] (3.1,-1.5) ellipse (1 and 0.325);
\node[node_style,inner sep=0.01cm] (7) at (2.7,-1.5) {$13$};
\node[node_style,inner sep=0.01cm] (7b) at (3.5,-1.5) {$14$};
\node[node_style,inner sep=0.01cm]  at (-3.5,-2.5) {};
\node[node_style,inner sep=0.01cm]  at (-2.7,-2.5) {};
\node[node_style,inner sep=0.01cm]  at (1.45,-2.5) {};
\node[node_style,inner sep=0.01cm]  at (2.7,-2.5) {};
\node[node_style,inner sep=0.01cm]  at (3.5,-2.5) {};
\node  at (0,-2.5) {\dots};
\node[node_style,inner sep=0.01cm,fill=inputcolor] (u) at (0,1) {$u$};

\draw[-stealth,semithick,line join=round,
semithick,
decorate, decoration={
    zigzag,
    segment length=4,
    amplitude=.9,post=lineto,
    post length=5pt
}] (u) -- (1);
\draw[-stealth,semithick,line join=round,
semithick,
decorate, decoration={
    zigzag,
    segment length=4,
    amplitude=.9,post=lineto,
    post length=5pt
}] (u) -- (1b);

\draw[-stealth,semithick] (1) to[in=205,out=155,loop,looseness=5] (1);
\draw[-stealth,semithick] (1b) to[in=-25,out=25,loop,looseness=5] (1b);

\draw[-stealth,semithick] (3) to[in=205,out=155,loop,looseness=5] (3);
\draw[-stealth,semithick] (3b) to[in=-25,out=25,loop,looseness=5] (3b);

\draw[-stealth,semithick] (2) to[in=205,out=155,loop,looseness=5] (2);
\draw[-stealth,semithick] (2b) to[in=-25,out=25,loop,looseness=5] (2b);

\draw[-stealth,semithick] (4) to[in=205,out=155,loop,looseness=5] (4);
\draw[-stealth,semithick] (4b) to[in=-25,out=25,loop,looseness=5] (4b);

\draw[-stealth,semithick] (5) to[in=205,out=155,loop,looseness=5] (5);
\draw[-stealth,semithick] (5b) to[in=-25,out=25,loop,looseness=5] (5b);

\draw[-stealth,semithick] (6) to[in=205,out=155,loop,looseness=5] (6);
\draw[-stealth,semithick] (6b) to[in=-25,out=25,loop,looseness=5] (6b);

\draw[-stealth,semithick] (7) to[in=205,out=155,loop,looseness=5] (7);
\draw[-stealth,semithick] (7b) to[in=-25,out=25,loop,looseness=5] (7b);

\draw[-stealth,semithick] (1) to[bend left=20] (1b);
\draw[-stealth,semithick] (1b) to[bend left=20] (1);
\draw[-stealth,semithick] (2) to[bend left=20] (2b);
\draw[-stealth,semithick] (2b) to[bend left=20] (2);
\draw[-stealth,semithick] (3) to[bend left=20] (3b);
\draw[-stealth,semithick] (3b) to[bend left=20] (3);
\draw[-stealth,semithick] (4) to[bend left=20] (4b);
\draw[-stealth,semithick] (4b) to[bend left=20] (4);
\draw[-stealth,semithick] (5) to[bend left=20] (5b);
\draw[-stealth,semithick] (5b) to[bend left=20] (5);
\draw[-stealth,semithick] (6) to[bend left=20] (6b);
\draw[-stealth,semithick] (6b) to[bend left=20] (6);
\draw[-stealth,semithick] (7) to[bend left=20] (7b);
\draw[-stealth,semithick] (7b) to[bend left=20] (7);
\draw[-stealth,semithick] (1) -- (2);
\draw[-stealth,semithick] (1b) -- (3);
\draw[-stealth,semithick] (2) to[bend right=20] (4);
\draw[-stealth,semithick] (2b) -- (5);
\draw[-stealth,semithick] (3) -- (6);
\draw[-stealth,semithick] (3b) to[bend left] (7);
\draw[dashed,semithick] (4) -- (-3.5,-2.25);
\draw[dashed,semithick] (4b) -- (-2.7,-2.25);
\draw[dashed,semithick] (5) -- (-1.45,-2.25);
\draw[dashed,semithick] (5b) -- (-0.65,-2.25);
\draw[dashed,semithick] (6b) -- (1.45,-2.25);
\draw[dashed,semithick] (6) -- (0.65,-2.25);
\draw[dashed,semithick] (7b) -- (3.5,-2.25);
\draw[dashed,semithick] (7) -- (2.7,-2.25);

\node at (-3.5,0.75) {(b)};

\end{scope}
\end{tikzpicture}
}
\end{center}
\vspace{-1cm}
\caption{(a) Binary tree network of height $h$.  The red ellipses denote the nodes associated to a higher-order subsystem in ${\rm G}(\hat{\Sigma})$, which include all nodes of the tree except for the leaves (terminal nodes). The state nodes $x_i\in\mathcal{V}$ are denoted simply by $i$ to reduce cluttering.}
\vspace{-0.25cm}
\label{fig:tree}
\end{figure}
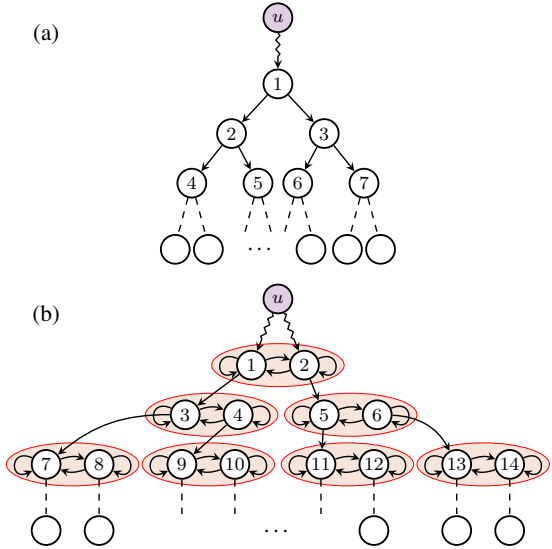
This system is not structurally controllable by Theorem \ref{thm:gen_dim}, since the maximum-cardinality set of nodes $\Zc$ spanned vertex-disjoint stems and cycles is a stem with $h+1$ nodes in $\Vc$. By Theorem \ref{thm:gen_dim}, to make the network structurally controllable, the minimum number of first-order heterogeneous subsystems that have to be added is $\sum_{i=0}^{h} 2^{i} - (h+1)=2^{h+1}-(h+2)$, namely one for each subsystem whose corresponding node in ${\rm G}(\Sigma)$ does not belong to $\Zc$.
Consider the extended binary tree network of Fig.~\ref{fig:tree}(b) which is obtained by replacing each first-order subsystem in ${\Sigma}$ with an higher-order subsystem of dimension $\hat{n}_M=2$ except for subsystems associated to the $2^{h}$ leaves (terminal nodes). The extended network system $\hat{\Sigma}$ contains $\hat{S}=\sum_{i=0}^{h-1} 2^{i}=2^{h}-1$ higher-order heterogeneous subsystems.

With reference to the labeling of nodes of Fig.~\ref{fig:tree}(b), the state matrix $\hat{A}$ of the extended network has the block lower triangular form
    
    \vspace{-10px}
    \begin{align*}
        \hat{A}=\left[\begin{array}{c|c}
        A^{(1)} & 0\\
        \hline 
        A^{(2)} & 0
         \end{array}\right].
    \end{align*}\vspace{-2px}
Here, 
$
A^{(2)}=\left[\begin{array}{ccccc} 0 & I_{\star}
 \end{array}\right]\in\mathbb{R}^{2^{h}\times(2^{h+1}-2)},
$
with $I_{\star}$ being a $2^{h}\times 2^{h}$ diagonal matrix with generic diagonal entries, contains the weights of the edges pointing to the terminal nodes of the tree and 
 $A^{(1)}\in\mathbb{R}^{(2^{h+1}-2)\times(2^{h+1}-2)}$ is a block $2\times 2$ lower triangular matrix
$$
A^{(1)}=\left[\begin{array}{ccccc}A^{(1)}_{11} & 0 & \cdots & 0\\ 
A^{(1)}_{21} & A^{(1)}_{22} & \ddots &  \vdots\\
\vdots & \ddots& \ddots & 0\\
A^{(1)}_{2^{h}-1,1} & \cdots& \cdots & A^{(1)}_{2^{h}-1,2^{h}-1}\!\!
 \end{array}\right].
$$
where\footnote{The entries marked with $\star$ are considered as generic.}\vspace{-2px}
\begin{align}\label{eqn:tree_off_diag}
A^{(1)}_{ij}&=\begin{cases}
\left[\begin{smallmatrix}
    \star & \star \\ \star & \star
\end{smallmatrix}\right] & \text{ if } i=j,\\[0.1cm]
\left[\begin{smallmatrix}\star & 0 \\ 0 & 0\end{smallmatrix}\right] & \text{ if } i=\min \{i: (j,i)\in\mathcal{E}_v \},\\[0.1cm]
\left[\begin{smallmatrix}0& \star \\ 0 & 0\end{smallmatrix}\right]& \text{ if }  i=\max \{i: (j,i)\in\mathcal{E}_v \},\\[0.1cm] 
0 & \text{ otherwise}.
\end{cases}
\end{align}\vspace{-2px}

Moreover, the input matrix $\hat B$ is such that $\hat b_1\in\mathbb{R}^2$ is generic and $\hat b_i=0$ for $i\ne 1$, and the output matrix $\hat C$ is such that $\hat  c_i\in\mathbb{R}^{1\times 2}$  have generic entries for $i\in\{1,\dots,2^h-1\}$  and $\hat c_i=1$ otherwise.

 The following proposition shows that the extended binary tree network is structurally output controllable, which implies that $\Delta= 2^h-(h+1)$. Note that, if $h\geq2$, then $\Delta>0$, confirming the advantage of introducing higher-order heterogeneous subsystems for this class of Y-networks.

\begin{proposition}
    The extended binary tree network described above (see Fig. \ref{fig:tree}(b)) is structurally output controllable.
\end{proposition}
\begin{proof}
Consider a numerical realization $\bar{A}$ of $\hat{A}$ where the diagonal blocks follow the zero pattern  
$$
A^{(1)}_{ii} = \left[\begin{smallmatrix}
    0 & \star \\ \star & 0
\end{smallmatrix}\right]
$$
and the off-diagonal entries are selected to ensure that each $A^{(1)}_{ii}$ has distinct nonzero eigenvalues. Moreover, assign nonzero values to the remaining generic entries of $\hat{A}$ in its numerical realization $\bar{A}$. Let $\bar{B}$ be a numerical realization of $\hat{B}$ such that $\hat{b}_1 = [1\ \ 0]^\top$. Similarly, define $\bar{C}$ as a numerical realization of $\hat{C}$, where  $\hat{c}_i = [1\ \ 0]$ for $i \in \{1, \dots, 2^h - 1\}$.  

In view of the block triangular structure of $\bar A$ and the conditions on its diagonal blocks, it follows that (i) $\bar A$ is diagonalizable, and (ii) $\bar\lambda=0$ is the only uncontrollable eigenvalue of the system. Therefore, we are in position to apply Corollary \ref{cor:PBH} in Appendix \ref{sec:out_PBH}. We note that

\begin{align*}
\rank(\bar{C}[\bar\lambda I-\bar{A} \ \ \bar{B}])=\rank(\bar{C}[\bar{B} \ \ \bar{A})].
\end{align*}
Moreover, $\bar{C}[\bar{B} \ \ \bar{A}]=[T \ \ 0]$ where $T\in\mathbb{R}^{(2^{h+1} -1)\times(2^{h+1} -1)}$ is an upper triangular matrix with nonzero diagonal entries. Thus, $\bar{C}[\bar\lambda I-\bar{A} \ \ \bar{B}]$ is full row-rank and the system described by $(\bar A, \bar B,\bar C)$ is output controllable. Since there exist a numerical realization of $\hat{A},\hat{B},\hat{C}$ that makes $\hat{\Sigma}$ output controllable, $\hat{\Sigma}$ is structurally output controllable. 
\end{proof}

\subsubsection{Case study II. Single bifurcation network}
 Consider the structured system described by the single bifurcation network ${\rm {G}}(\Sigma)$ with height $h$ of Fig.~\ref{fig:bifurcation}(a) and assume that $h$ is even.
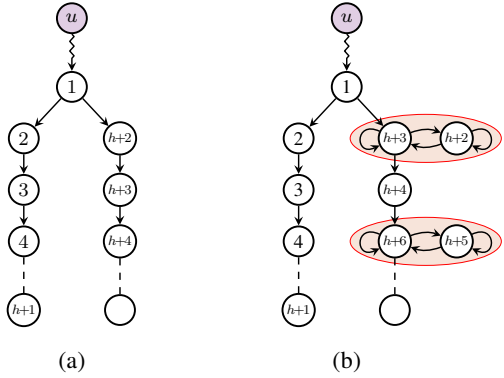
\begin{figure}[!h]
\begin{center}
\resizebox{0.375\textwidth}{!}{
\begin{tikzpicture}[node_style/.style={draw, circle,thick, fill=white, minimum size=0.435cm,font=\footnotesize}]
\begin{scope}   
\node[node_style,inner sep=0.01cm] (1) at (0,0) {$1$};
\node[node_style,inner sep=0.01cm] (2) at (-0.7,-0.75) {$2$};
\node[node_style,inner sep=0.01cm] (3) at (-0.7,-1.5) {$3$};
\node[node_style,inner sep=0.01cm] (4) at (-0.7,-2.25) {$4$};
\node[node_style,scale=0.7,inner sep=0.05cm] (h) at (-0.7,-3.25) {$h\!\!+\!\!1$};
\node[node_style,scale=0.7,inner sep=0.05cm] (h2) at (0.7,-0.75) {$h\!\!+\!\!2$};
\node[node_style,scale=0.7,inner sep=0.05cm] (h3) at (0.7,-1.5) {$h\!\!+\!\!3$};
\node[node_style,scale=0.7,inner sep=0.05cm] (h4) at (0.7,-2.25) {$h\!\!+\!\!4$};
\node[node_style,inner sep=0.01cm] (n) at (0.7,-3.25) {\ };

\node[node_style,inner sep=0.01cm,fill=inputcolor] (u) at (0,1) {$u$};


\draw[-stealth,semithick,line join=round,
semithick,
decorate, decoration={
    zigzag,
    segment length=4,
    amplitude=.9,post=lineto,
    post length=5pt
}] (u) -- (1);

\draw[-stealth,semithick] (1) -- (2);
\draw[-stealth,semithick] (2) -- (3);
\draw[-stealth,semithick] (3) -- (4);
\draw[dashed,semithick] (4) -- (h);
\draw[-stealth,semithick] (1) -- (h2);

\draw[-stealth,semithick] (h2) -- (h3);
\draw[-stealth,semithick] (h3) -- (h4);
\draw[dashed,semithick] (h4) -- (n);

\node at (0,-4) {(a)};
\end{scope}
\begin{scope}[xshift=4cm]
\node[node_style,inner sep=0.01cm] (1) at (0,0) {1};
\node[node_style,inner sep=0.01cm] (2) at (-0.7,-0.75) {2};
\node[node_style,inner sep=0.01cm] (3) at (-0.7,-1.5) {3};
\node[node_style,inner sep=0.01cm] (4) at (-0.7,-2.25) {4};
\node[node_style,scale=0.7,inner sep=0.05cm] (h) at (-0.7,-3.25) {$h\!\!+\!\!1$};
\draw[draw=red,line width=0.01,fill=BrickRed!10] (1.15,-0.75) ellipse (1.1 and 0.35);
\node[node_style,scale=0.7,inner sep=0.05cm] (h2) at (0.7,-0.75) {$h\!\!+\!\!3$};
\node[node_style,scale=0.7,inner sep=0.05cm] (h2b) at (1.6,-0.75) {$h\!\!+\!\!2$};
\node[node_style,scale=0.7,inner sep=0.05cm] (h3) at (0.7,-1.5) {$h\!\!+\!\!4$};
\draw[draw=red,line width=0.01,fill=BrickRed!10] (1.15,-2.25) ellipse (1.1 and 0.35);
\node[node_style,scale=0.7,inner sep=0.05cm] (h4) at (0.7,-2.25) {$h\!\!+\!\!6$};
\node[node_style,scale=0.7,inner sep=0.05cm] (h4b) at (1.6,-2.25) {$h\!\!+\!\!5$};
\node[node_style,inner sep=0.01cm] (n) at (0.7,-3.25) {\ };

\node[node_style,inner sep=0.01cm,fill=inputcolor] (u) at (0,1) {$u$};


\draw[-stealth,semithick,line join=round,
semithick,
decorate, decoration={
    zigzag,
    segment length=4,
    amplitude=.9,post=lineto,
    post length=5pt
}] (u) -- (1);

\draw[-stealth,semithick] (1) -- (2);
\draw[-stealth,semithick] (2) -- (3);
\draw[-stealth,semithick] (3) -- (4);
\draw[dashed,semithick] (4) -- (h);
\draw[-stealth,semithick] (1) -- (h2);
\draw[-stealth,semithick] (h2) -- (h3);
\draw[-stealth,semithick] (h3) -- (h4);
\draw[dashed,semithick] (h4) -- (n);
\draw[-stealth,semithick] (h2) to[bend left=20] (h2b);
\draw[-stealth,semithick] (h2b) to[bend left=20] (h2);
\draw[-stealth,semithick] (h4) to[bend left=20] (h4b);
\draw[-stealth,semithick] (h4b) to[bend left=20] (h4);

\draw[-stealth,semithick] (h2) to[in=205,out=155,loop,looseness=5] (h2);

\draw[-stealth,semithick] (h2b) to[in=-25,out=25,loop,looseness=5] (h2b);

\draw[-stealth,semithick] (h4) to[in=205,out=155,loop,looseness=5] (h4);

\draw[-stealth,semithick] (h4b) to[in=-25,out=25,loop,looseness=5] (h4b);

\node at (0,-4) {(b)};
\end{scope}
\end{tikzpicture}
}
\end{center}
\vspace{-1cm}
\caption{(a) Single bifurcation network of height $h$.  (b) Extended single bifurcation network. The red ellipses denote the extended nodes, which include the even nodes of the right branch of the original network ($h$ is assumed to be even). The state nodes $x_i\in\mathcal{V}$ are denoted simply by $i$ to reduce cluttering.}
\vspace{-0.25cm}
\label{fig:bifurcation}
\end{figure}
This system is not structurally controllable by Theorem \ref{thm:gen_dim}, since the maximum-cardinality set of nodes $\Zc$ spanned by vertex-disjoint stems and cycles is a stem with $h+1$ nodes in $\Vc$. To make the network structurally controllable, the minimum number of first-order heterogeneous subsystems that have to be added is $h$.
Consider the extended bifurcation network of Fig.~\ref{fig:bifurcation}(b) which is obtained by replacing each first-order subsystem associated with the even nodes of the right branch of ${\rm{G}}(\Sigma)$ with a second-order subsystem. The extended network system $\hat{\Sigma}$ contains $\hat{S}=h/2$~higher-order~heterogeneous~subsystems. 

With reference to the labeling of nodes of Fig.~\ref{fig:bifurcation}(b), the state matrix $\hat{A}$ of the extended network has the block lower triangular form
$$
A=\left[\begin{array}{c|c}
A^{(11)} & 0\\
\hline 
A^{(21)} & A^{(22)}
 \end{array}\right].
$$
Here, $A^{(11)}\in\mathbb{R}^{(h+1)\times (h+1)}$ is the transpose of a nilpotent Jordan block with nonzero entries set to generic values, $A^{(21)}$$\in\mathbb{R}^{3h/2\times (h+1)}$ has only one generic entry in position $(2,1)$, and $A^{(22)}\in\mathbb{R}^{{{3h/2}\times {3h/2}}}$ is a block $3\times 3$ triangular matrix of the form\vspace{-2px}
$$
A^{(22)}=\left[\begin{array}{ccccc}A^{(22)}_{11} & 0 & \cdots & 0\\ 
A^{(22)}_{21} & A^{(22)}_{22} & \ddots &  \vdots\\
\vdots & \ddots& \ddots & 0\\
A^{(22)}_{h/2} & \cdots& \cdots & A^{(22)}_{h/2,h/2}\!
 \end{array}\right],
$$\vspace{-2px}
where\vspace{-2px}
\begin{align}\label{eqn:bif_off_diag}
A^{(22)}_{ij}&=\begin{cases}\left[\begin{smallmatrix} \star & \star  & 0 \\ \star & \star & 0 \\   0 & \star & 0\end{smallmatrix}\right] & \text{ if } i=j,\\[0.2cm]
\left[\begin{smallmatrix}0 & 0  & 0 \\ 0 & 0 & \star \\ 0 & 0 & 0\end{smallmatrix}\right]  & \text{ if } i=j+1,\\
 0 & \text{ otherwise}.
\end{cases}
\end{align}
Moreover, the input matrix $\hat B$ is such that $\hat b_1$ is a generic real number and $\hat b_i=0$ for $i\ne 1$. The output matrix $\hat{C}$ is such that $\hat{c}_i \in \mathbb{R}^{1\times 2}$ { have generic entries for $i \in \{h+2, h+4, \dots\}$  and $\hat{c}_i=1$ otherwise.}

 The following proposition shows that the extended single bifurcation network is structurally output controllable, which implies that $\Delta= h/2>0$. This confirms the advantage of introducing higher-order heterogeneous subsystems for this class of Y-networks.\footnote{It is worth noting that this advantage is the maximum achievable since the number of modified subsystems equals the lower bound~of~Theorem~\ref{prop:bounds}.}
\begin{proposition}
    The extended single bifurcation network described above (see Fig. \ref{fig:bifurcation}(b)) is structurally output controllable.
\end{proposition}
\begin{proof}
Consider a numerical realization $\bar{A}$ of $\hat{A}$ where the diagonal blocks of $A^{(22)}$ follow the zero pattern  
$$
A^{(22)}_{ii} =
\left[\begin{smallmatrix}
    0 & \star  & 0 \\ \star & 0 & 0 \\   0 & \star & 0
\end{smallmatrix}\right]
$$
and the entries in positions $(1,2)$ and $(2,1)$ are selected to ensure that each $A^{(22)}_{ii}$ has two distinct nonzero eigenvalues. Moreover, assign nonzero values to the remaining generic entries of $\hat{A}$ in its numerical realization $\bar{A}$. Let $\bar{B}$ be a numerical realization of $\hat{B}$ such that $\hat{b}_1 = 1$. Similarly, define $\bar{C}$ as a numerical realization of $\hat{C}$, where  $\hat{c}_i = [0\ \ 1]$ for $i \in \{h+2, h+4, \dots\}$. 

In the above numerical realization $\bar A$, the block $A^{(11)}$ is not diagonalizable since it has only one eigenvalue $\bar\lambda=0$ with algebraic multiplicity $h+1$ and geometric multiplicity $1$. 
Moreover, the block $A^{(22)}$ has $h$ distinct nonzero eigenvalues and one eigenvalue $\bar\lambda=0$ with algebraic and  geometric multiplicity equal to  $h/2$.\footnote{The fact that the geometric multiplicity equals $h/2$ follows from $\mathrm{rank}\, A^{(22)}=h$, which yields $\dim \ker A^{(22)} = 3h/2- \mathrm{rank}\, A^{(22)} = h/2$.} Therefore, for such choice of weights $A^{(22)}$ is diagonalizable.

From the above observations, it can be shown that (i) $\bar A^{\top}$ has only one non-trivial Jordan chain associated with the eigenvalue $\bar\lambda=0$ and the ordinary eigenvectors of this chain have the form $w=[\alpha \ 0 \ \cdots \ 0]^{\top}$, $\alpha\ne 0$, and (ii) $\bar\lambda=0$ is the only uncontrollable eigenvalue of the system.

Since $\mathrm{span}\{w\}\cap \ker \bar B^{\top}=\{0\}$ and $(\bar A, \bar B)$ has only one uncontrollable eigenvalue $\bar{\lambda} =0$, we are in position to apply Corollary \ref{cor:PBH} in Appendix \ref{sec:out_PBH}. We note that
\begin{align*}
\mathrm{rank}(\bar{C}[\bar\lambda I-\bar{A} \ \ \bar{B}])&=\mathrm{rank}(\bar{C}[\bar{B} \ \ \bar{A}]).
\end{align*}
Moreover, by deleting the columns of $\bar{C}[\bar{B} \ \ \bar{A}]$ in position $h+2+3k$, $k\in\{0,1,2,\dots,h/2-1\}$, we obtain a square matrix $T\in\mathbb{R}^{(2h+1)\times(2h+1)}$ which is lower triangular with nonzero diagonal entries. The existence of a $(2h+1)\times(2h+1)$ nonsingular submatrix of $\bar{C}[\bar{B} \ \ \bar{A}]$ implies that the latter matrix has full row-rank. 

Thus, $\bar{C}[\bar\lambda I-\bar{A} \ \ \bar{B}]$ is full row-rank and the system described by $(\bar A, \bar B,\bar C)$ is output controllable. Since there exists a numerical choice of $\hat{A},\hat{B},\hat{C}$ that makes $\hat{\Sigma}$ output controllable, $\hat{\Sigma}$ is~structurally~output~controllable. 
\end{proof}

\subsection{Making general networks structurally controllable} \label{sec:gen_net}
Summarizing, so far we have that:
\begin{enumerate}
    \item X-networks can be modified to obtain a structurally controllable network (using the procedure detailed in Theorem \ref{thm:hoi}).
    \item Y-networks cannot be made neither structurally controllable nor structurally output controllable with the same approach, as the introduction of heterogeneous higher-order dynamics is needed. However,  we showed that it is still possible to obtain advantages in term of the number of subsystems to be modified for structural output controllability with respect to the introduction of heterogeneous first-order dynamics.
\end{enumerate}

 Input-accessible network systems that lack structural controllability and do not fall under type X or Y can be modified by introducing a mixture of homogeneous and heterogeneous higher-order subsystems to achieve structural (output) controllability, potentially obtaining a combination of the benefits outlined in Section \ref{sec:advantages}. A method to render networks structurally (output) controllable via higher-order dynamics can be obtained by first using the approach of Theorem \ref{thm:hoi}, and subsequently the same ideas we used for Y-networks on the already modified graph. A detailed exposition is provided next.

Let ${\Sigma}$ be an input-accessible network system which is not structurally controllable. $\Sigma$ satisfies the hypothesis of Corollary \ref{cor:ctrb_cases}, therefore every set of stem and cycles covering all state nodes in ${\rm G}(\Sigma)$ contains non-vertex disjoint paths.

We consider a (fixed) set of stems and cycles $\Pc$ covering all state nodes in $\textrm{G} (\Sigma)$. To design an extended system which is structurally controllable we perform the following steps.

\textbf{Step 1:} We apply the graph-theoretic method specified in the proof of Theorem \ref{thm:hoi}, considering the set of paths $\Pc$ fixed above: { we identify all state nodes corresponding, to intersections of type 1)-2)-3), {\em with the addition of the {state} nodes that correspond to intersections of stems that are originated by the same input node.}} After applying the method we obtain an extended system $\hat{\Sigma},$ and a set of stems and cycles $\hat{\Pc}$ covering all state node in ${\rm G}(\hat{\Sigma})$.

The design performed in this step guarantees that the cycles in $\hat{\Pc}$ are vertex-disjoint. However, if the system $\Sigma$ is not of type X, then $\hat{\Pc}$ contains stems originated from the same input node. Therefore, $\hat{\Pc}$ is not a set of vertex-disjoint stems and cycles and we cannot conclude that the system is structurally controllable.

\textbf{Step 2:} We consider $\hat{\Sigma}$ obtained in the previous step, and a set of vertex-disjoint stems $\hat{\Sc}\subseteq \hat{\Pc}$ covering the maximum number of { state} nodes in ${\rm G}(\hat{\Sigma})$. Let $\Zc$ be the set of state nodes belonging to {\em stems} in $\hat{\Pc}\setminus \hat{\Sc}$. Next, we proceed as follows:\\ (i) If a node in $\Zc$ is associated to a higher-order subsystem in $\hat{\Sigma}$,  we add {\em heterogeneous} internal dynamics to that subsystem (notice that after step 1 all subsystems in $\hat{\Sigma}$ are homogeneous). In the new graph ${\rm G}(\hat{\Sigma})$, we keep the same node labeling we had before this last modification;\\ (ii) If a node in $\Zc$ is associated to a first-order subsystem in $\hat{\Sigma}$, we modify the corresponding subsystem such that it has first-order {\em heterogeneous} dynamics, while leaving its label and connections to other nodes as they were. 
 
 After Step 2, all state nodes in $\hat{\Pc}\setminus \hat{\Sc}$ are covered by vertex-disjoint cycles, either self-loops, due to the added heterogeneous dynamics, or cycles already present in $\hat{\Pc}$. Since the remaining state nodes in $\hat{\Sc}$ are covered by vertex-disjoint stems, we conclude that, after Step 2, all state nodes are covered by vertex-disjoint stems and cycles.
Therefore, by Theorem \ref{thm:gen_dim} the system $\hat{\Sigma}$ designed through the outlined procedure is structurally controllable. {By Lemma \ref{lem:str_out_ctrb}, this in turn implies that $\hat{\Sigma}$ structurally output controllable}.

\section{Conclusion}\label{sec:conclusion}

This work explores the potential of introducing higher-order local dynamics in networks of identical, first-order systems to attain structural controllability properties. After establishing a connection between state controllability in networks of first-order systems and output controllability in networks of higher-order systems, we discuss how structural output controllability can be achieved by properly replacing a subset of nodes in the original network with higher-order systems. A known result in the literature is that structural controllability in input-accessible networks can be attained by allowing its nodes to have first-order heterogeneous internal dynamics. We show that introducing higher-order systems offers two key advantages with respect to first-order dynamics: A1) Structural output controllability can be ensured while preserving the homogeneity of the nodes, as they were in the original network; A2) Fewer subsystems need to be modified to achieve structural output controllability.

We characterize the class of X-networks, where both advantages can be achieved, and provide a method to replace only the controllability-preventing nodes with higher-order systems having homogeneous internal dynamics. This approach leads to potentially unbounded advantages in required modifications as network size grows.
We next define the class of Y-networks, where A1) cannot be achieved. We establish bounds on the number of higher-order heterogeneous subsystems that must be introduced to achieve structural output controllability. We also derive a PBH-like test, that allow us to prove that A2) is still achievable in certain scalable Y-network topologies like binary trees and bifurcation networks.
Our framework extends to general networks beyond X or Y classes, enabling output controllability with the introduction of a reduced number of homogeneous and heterogeneous subsystems.

Several open questions remain for further investigation, in particular the development of a systematic approach to achieve structural output controllability in Y-networks while minimizing the number of subsystem modifications. 
Moreover, the analysis performed in this paper can be combined with the one in \cite{MP-GB-FT:2024} to improve both structural and energy-related controllability by introducing higher-order local dynamics.

\appendix

\subsection{PBH tests for output controllability}\label{sec:out_PBH}

To state and prove the main results of this section, we introduce some additional notation.
Given a matrix $A\in\mathbb{R}^{n\times m}$, we let $\im(A)$ be image of $A$. The symbols $\oplus$ and $\otimes$ stand for the direct sum of subspaces and the Kronecker product of matrices, respectively. We indicate with $\mathbf{1}_n$ the $n$-dimensional vector of all ones and with $\mathcal{X}^\perp$ the orthogonal complement of a subspace $\mathcal{X}$. Given matrices $A_i\in\mathbb{R}^{n_i\times m_i}$, $i=1,\dots,N$, $\mathrm{diag}(A_1,\dots,A_N)$ denotes the block matrix with diagonal blocks $A_1,\dots,A_N$. For a square matrix $A\in\mathbb{R}^{n\times n}$, we let $\{\lambda_{i}\}_{i=1}^{r}$ be the eigenvalues of $A$ and $\nu_{i}$ be the algebraic multiplicity of $\lambda_{i}$ in the minimal polynomial of $A$, i.e., the maximum size of Jordan blocks corresponding to $\lambda_{i}$. Moreover, we let $\{w_{i}\}_{i=1}^{\ell}$ be the (ordinary) eigenvectors of $A^{\top}$ corresponding to non-trivial Jordan chains, i.e., Jordan chains of length greater or equal than $2$, and we define $\mathcal{N}:=\mathrm{span}\{w_{1},\dots,w_{\ell}\}$. 
Note that for any $x\in\mathbb{R}^{n}$ which is a generalized eigenvector of $A^{\top}$, i.e., satisfies $(\lambda_{i}I-A^{\top})^{q+1}x=0$ and $(\lambda_{i}I-A^{\top})^{q}x\ne 0$ for some $i\in\{1,\dots,r\}$ and $q\in\{1,\dots,\nu_{i}-1\}$, we have $(\lambda_{i}I-A^{\top})^{q}x\in\mathcal{N}$.

\begin{theorem}\label{thm:PBH}
Consider a linear system as in \eqref{eq:netsys} with $p< n$ and $C$ full row-rank. Assume that $\mathcal{N}\cap \ker B^{\top}\subseteq \mathrm{Im}\, C^{\top}$. Then, the following statements are equivalent
\begin{enumerate}[(i)]
\item the system is output controllable.
\item $\mathrm{Im}\, C^{\top} \cap\left( \bigoplus_{i=1}^{r}  \ker (\lambda_{i}I-A^{\top}) \cap \ker B^{\top}\right)=\{0\}$. 
\item The matrix $$\begin{bmatrix} \mathrm{diag}([\lambda_{1}I-A\ \ B],\dots,[\lambda_{r}I-A\ \ B]) &  \mathbf{1}_{r}\!\otimes\! K_{C} \end{bmatrix}$$
 has full row-rank, where $K_{C}\in\mathbb{R}^{n\times(n-p)}$ is such that $\mathrm{Im}\, K_{C} =\ker C$.
\end{enumerate}
\end{theorem}
\begin{proof}
[(i) $\Leftrightarrow$ (ii)]. From \cite[Theorem 3.1]{BD-JL-MJ:23}, the system is output controllable if and only if
\vspace{-2px}
\begin{equation}\label{eq:gen-cond}
\mathrm{Im}\, C^{\top} \cap \bigoplus_{i=1}^{r} E_{i}=\{0\},
\end{equation} 
\vspace{-2px}
where $$E_{i}:= \ker(\lambda_{i}I-A^{\top})^{\nu_{i}}\cap \left(\bigcap_{k=0}^{\nu_{i}-1}\ker\left(B^{\top}(\lambda_{i}I-A^{\top})^k\right)\right).$$
We will show that (ii) is equivalent to \eqref{eq:gen-cond} when $\mathcal{N}\cap \ker B^{\top}\subseteq \mathrm{Im}\, C^{\top}$. More precisely, we will prove 
\begin{gather*}
    \mathrm{Im}\, C^{\top} \cap \bigoplus_{i=1}^{r} E_{i}\ne \{0\}\\
    \Updownarrow\\[-0.25cm]
    \mathrm{Im}\, C^{\top} \cap\left( \bigoplus_{i=1}^{r}  \ker (\lambda_{i}I-A^{\top}) \cap \ker B^{\top}\right)\neq\{0\},
\end{gather*}
assuming that $\mathcal{N}\cap \ker B^{\top}\subseteq \mathrm{Im}\, C^{\top}$.

First, note that, for all $i\in\{1,\dots,r\}$
$$
\ker (\lambda_{i}I-A^{\top}) \cap \ker B^{\top}\subseteq E_{i}.
$$
Hence, $\mathrm{Im}\, C^{\top} \cap\left( \bigoplus_{i=1}^{r} \ker B^{\top} \cap \ker (\lambda_{i}I-A^{\top})\right)\neq\{0\}$ implies $\mathrm{Im}\, C^{\top} \cap \bigoplus_{i=1}^{r} E_{i}\neq\{0\}$.

Conversely, suppose that  $\mathrm{Im}\, C^{\top} \cap \bigoplus_{i=1}^{r} E_{i}\neq\{0\}$. Then there exists $i\in\{1,\dots,r\}$ such that $E_{i}\ne \{0\}$. Note that for all $x\in E_{i}\setminus \{0\}$ there exists $q\in\{0,\dots,\nu_{i}-1\}$ such that $(\lambda_{i}I-A^{\top})^{q}x\ne 0$ and $(\lambda_{i}I-A^{\top})^{q}x\in \ker (\lambda_{i}I-A^{\top}) \cap \ker B^{\top}$. In particular, if $q\ne 0$ then $x$ is a generalized eigenvector of $A^{\top}$ and, by definition of $\mathcal{N}$, it holds $(\lambda_{i}I-A^{\top})^{q}x\in \mathcal{N} \cap \ker B^{\top}\subseteq \ker (\lambda_{i}I-A^{\top}) \cap \ker B^{\top}$.
From the assumption $\mathcal{N}\cap \ker B^{\top}\subseteq \mathrm{Im}\, C^{\top}$, we conclude that 
$$ 
(\lambda_{i}I-A^{\top})^{q}x\in \mathrm{Im}\, C^{\top} \cap\left(\ker (\lambda_{i}I-A^{\top}) \cap \ker B^{\top}\right),
$$
which implies \vspace{-4px}
$$
 \mathrm{Im}\, C^{\top}\cap\left( \bigoplus_{i=1}^{r}  \ker (\lambda_{i}I-A^{\top}) \cap \ker B^{\top}\right)\ne \{0\}.
$$
[(ii) $\Leftrightarrow$ (iii)]. The following equivalences hold
\begin{align*}
&\mathrm{Im}\, C^{\top} \cap\left( \bigoplus_{i=1}^{r} \ker B^{\top} \cap \ker (\lambda_{i}I-A^{\top})\right)=\{0\}\\
\Leftrightarrow\ & \ker K_{C}^{\top} \cap\bigoplus_{i=1}^{r} \ker \begin{bmatrix}\lambda_{i}I-A^{\top} \\ B^{\top} \end{bmatrix}=\{0\}\\
\Leftrightarrow\ & \ker (\mathbf{1}_{r}^{\top}\!\otimes\! K_{C}^{\top}) \cap \ker\mathrm{diag}\!\left(\!\begin{bmatrix}\lambda_{i}I-A^{\top}\\ B^{\top}\end{bmatrix}_{i=1,\dots,r}\!\right)\!=\!\{0\}\\
\Leftrightarrow\ & \ker\begin{bmatrix}\mathrm{diag}\!\left(\!\begin{bmatrix}\lambda_{i}I-A^{\top}\\ B^{\top}\end{bmatrix}_{i=1,\dots,r}\!\right)\\ \mathbf{1}_{r}^{\top}\!\otimes\! K_{C}^{\top}\end{bmatrix}\!=\!\{0\}\\
\Leftrightarrow\ & \mathrm{rank}\begin{bmatrix} \mathrm{diag}([\lambda_{i}I-A\ \ B]_{i=1,\dots,r}) &  \mathbf{1}_{r}\!\otimes\! K_{C} \end{bmatrix}=rn,
\end{align*}
where in the first equivalence we used the identity $\mathrm{Im}\, C^{\top} = [\ker C]^{\perp} = [\mathrm{Im}\, K_{C}]^{\perp} = \ker K_{C}^{\top}$ and in the second equivalence the fact that the subspaces $\ker \left[\begin{smallmatrix}\lambda_{i}I-A^{\top} \\ B^{\top} \end{smallmatrix}\right]$, $i\in\{1,\dots,r\}$, are mutually disjoint.
\end{proof}

Some comments on Theorem \ref{thm:PBH} are in order. First, when $A$ is diagonalizable then $\mathcal{N}=\{0\}$ and the assumption $\mathcal{N}\cap \ker B^{\top}\subseteq \mathrm{Im}\, C^{\top}$ is always satisfied. Second, condition (iii) cannot in general be replaced by 
\begin{align}\label{eq:PBH-output}
\mathrm{rank}\, C [\lambda_{i}I-A\ \ B]=p, \ \ i\in\{1,\dots,r\}.
\end{align}
As a counterexample, consider the system described by
$
A=\left[\begin{smallmatrix}
    1 & 0 \\ 0 & 0
\end{smallmatrix}\right]$, $B=\left[\begin{smallmatrix}0  \\ 0 \end{smallmatrix}\right]$, $C=\left[\begin{smallmatrix}1  & 1 \end{smallmatrix}\right]$. 
In this case $A$ is diagonalizable with eigenvalues $\lambda_{1}=0$ and $\lambda_{1}=1$ and the system is clearly not output controllable (the input matrix is zero). However, $\mathrm{rank}\, C [\lambda_{1}I-A\ \ B]=\mathrm{rank}\, C [\lambda_{2}I-A\ \ B]=1$.

The next  corollary of Theorem \ref{thm:PBH} provides a condition under which output controllability~is~equivalent~to~\eqref{eq:PBH-output}.
\begin{corollary}\label{cor:PBH}
Consider a linear system as in \eqref{eq:netsys} with $p< n$ and $C$ full row-rank. 
Assume that $\mathcal{N}\cap \ker B^{\top}\subseteq \mathrm{Im}\, C^{\top}$ and the system has only one uncontrollable eigenvalue $\bar{\lambda}\in \{\lambda_{1},\dots,\lambda_{r}\}$. Then, the system is output controllable iff
\begin{align}\label{eq:PBH-output2}
\mathrm{rank}\, C [\bar{\lambda}I-A\ \ B]=p.
\end{align}
\end{corollary} 
\begin{proof}
If $\bar \lambda$ is the only uncontrollable eigenvalue of the system, then $\mathrm{rank}\begin{bmatrix}\lambda I-A & B\end{bmatrix}<n$ if and only if $\lambda=\bar\lambda$.  Thus, condition (iii) of  Theorem \ref{thm:PBH} reduces to 
\begin{align*}
&\mathrm{rank}\begin{bmatrix}\bar{\lambda} I-A & B & K_{C}\end{bmatrix} = n
\!\Leftrightarrow\!\ker\begin{bmatrix}\bar{\lambda} I-A^{\top} \\ B^{\top} \\ K_{C}^{\top}\end{bmatrix}=\{0\} \\
 &\Leftrightarrow \!\mathrm{Im}\, C^{\top} \!\cap \!\ker\begin{bmatrix}\bar{\lambda} I-A^{\top} \\ B^{\top}\end{bmatrix}\!=\!\{0\}
\Leftrightarrow\! \ker\!\begin{bmatrix}\bar{\lambda} I-A^{\top} \\ B^{\top}\end{bmatrix}C^{\top}\!=\!\{0\}\\
&\Leftrightarrow \mathrm{rank}\, C [\bar\lambda I-A\ \ B]=p,
\end{align*}~
where we used $\ker K_{C}^{\top}=\mathrm{Im}\, C^{\top}$.
\end{proof}

 \bibliographystyle{IEEEtran}
 \bibliography{IEEEabrv,bib-latex.bib}
\end{document}